\RequirePackage[l2tabu, orthodox]{nag}

\documentclass[11pt,reqno]{amsart}
\usepackage{fullpage,url,amssymb,enumerate,colonequals}
\usepackage[all]{xy} 
\usepackage{mathrsfs} 

\usepackage[dvipsnames,xcdraw,hyperref]{xcolor}

\newcommand{\defi}[1]{\emph{#1}} 
\newcommand{\cplusplus}{\texttt{C++}}

\newcommand{\C}{\mathbb{C}}
\newcommand{\CC}{\mathbb{C}}
\newcommand{\F}{\mathbb{F}}

\newcommand{\G}{\mathbb{G}}
\newcommand{\GG}{\mathbb{G}}

\newcommand{\Q}{\mathbb{Q}}
\newcommand{\QQ}{\mathbb{Q}}
\newcommand{\R}{\mathbb{R}}
\newcommand{\RR}{\mathbb{R}}
\newcommand{\Z}{\mathbb{Z}}
\newcommand{\ZZ}{\mathbb{Z}}

\newcommand{\Kbar}{{\overline{K}}}

\newcommand{\Fbar}{{\overline{\F}}}

\newcommand{\bb}{\mathbf{b}}

\newcommand{\vv}{\mathbf{v}}
\newcommand{\ww}{\mathbf{w}}

\newcommand{\pp}{\mathfrak{p}}

\newcommand{\cplxi}{\textrm{i}}

\usepackage{scalerel}
\newcommand\sangle[1][.5]{\mathbin{\ThisStyle{\vcenter{\hbox{%
  \scalebox{#1}{$\SavedStyle\angle$}}}}}%
}

\newcommand{\Angle}{\pmb{\sangle[1.5]}}
\newcommand{\Regge}{\mathfrak{R}}
\newcommand{\Sphere}{\Sigma}
\newcommand{\Tetrahedron}{\Delta}

\newcommand{\calF}{\mathcal{F}}

\newcommand{\calH}{\mathcal{H}}

\newcommand{\calL}{\mathcal{L}}

\newcommand{\calP}{\mathcal{P}}

\newcommand{\scriptA}{\mathscr{A}}

\newcommand{\scriptM}{\mathscr{M}}

\newcommand{\scriptR}{\mathscr{R}}

\DeclareMathOperator{\Aut}{Aut}

\DeclareMathOperator{\Spec}{Spec}

\DeclareMathOperator{\Stab}{Stab}

\DeclareMathOperator{\sym}{sym}

\newcommand{\tors}{{\operatorname{tors}}}

\newcommand{\GL}{\operatorname{GL}}

\newcommand{\M}{\operatorname{M}}

\newcommand{\directsum}{\oplus} 

\newcommand{\intersect}{\cap} 
\newcommand{\Intersection}{\bigcap} 
\newcommand{\isom}{\simeq}

\newcommand{\tensor}{\otimes} 
\newcommand{\To}{\longrightarrow}
\newcommand{\union}{\cup} 
\newcommand{\Union}{\bigcup} 

\newtheorem{theorem}{Theorem}[section]
\newtheorem{lemma}[theorem]{Lemma}
\newtheorem{corollary}[theorem]{Corollary}
\newtheorem{proposition}[theorem]{Proposition}

\theoremstyle{definition}
\newtheorem{definition}[theorem]{Definition}

\newtheorem{example}[theorem]{Example}

\newtheorem{algorithm}[theorem]{Algorithm}

\theoremstyle{remark}
\newtheorem{remark}[theorem]{Remark}

\makeatletter
\g@addto@macro\bfseries{\boldmath} 
\makeatother

\usepackage{tikz}
\usepackage{microtype}

\usepackage[
  pagebackref,
  pdfauthor={Kiran Kedlaya, Alexander Kolpakov, Bjorn Poonen, Michael Rubinstein}, 
  pdftitle={Space vectors forming rational angles}
]{hyperref}

\usepackage[alphabetic,backrefs,lite,nobysame]{amsrefs} 

\makeatletter
\@namedef{subjclassname@2020}{%
  \textup{2020} Mathematics Subject Classification}
\makeatother

\begin{document}

\title{Space vectors forming rational angles}

\author{Kiran S. Kedlaya}
\address{Department of Mathematics, University of California, San Diego,
  La Jolla, CA 92093, USA}
\email{kedlaya@ucsd.edu}
\urladdr{\url{https://kskedlaya.org/}}

\author{Alexander Kolpakov}
\address{\parbox{\linewidth}{Institut de Math\'ematiques, Universit\'e de Neuch\^atel, 2000 Neuch\^atel, Suisse/Switzerland \\
    Laboratory of combinatorial and geometric structures, Moscow Institute of Physics \\ and Technology, Dolgoprudny, Russia}}
\email{kolpakov.alexander@gmail.com}
\urladdr{\url{https://sashakolpakov.wordpress.com}}

\author{Bjorn Poonen}
\address{Department of Mathematics, Massachusetts Institute of Technology, Cambridge, MA 02139-4307, USA}
\email{poonen@math.mit.edu}
\urladdr{\url{http://math.mit.edu/~poonen/}}

\author{Michael Rubinstein}
\address{Pure Mathematics, University of Waterloo, Waterloo ON, N2L 3G1, Canada}
\email{mrubinst@uwaterloo.ca}
\urladdr{\url{http://www.math.uwaterloo.ca/~mrubinst/}}

\dedicatory{In memory of John H. Conway}

\thanks{K.S.K.\ was supported in part by National Science Foundation grant DMS-1802161 and the UCSD Warschawski Professorship. A.K.\ was supported in part by the Swiss National Science Foundation project PP00P2-170560 and by the Russian Federation Government (grant no. 075-15-2019-1926). B.P.\ was supported in part by National Science Foundation grant DMS-1601946 and Simons Foundation grants \#402472 (to Bjorn Poonen) and \#550033. M.R.\ was supported by an NSERC Discovery Grant.}

\date{November 28, 2020}

\begin{abstract}
  We classify all sets of nonzero vectors in $\R^3$ such that
  the angle formed by each pair is a rational multiple of $\pi$.
  The special case of four-element subsets
  lets us classify all tetrahedra whose dihedral angles
  are multiples of $\pi$, solving$\!$ a$\!$ 1976$\!$ problem$\!$ of$\!$ Conway$\!$ and$\!$ Jones:
  there are $2$ one-parameter families and $59$ sporadic tetrahedra,
  all but three of which are related to either the icosidodecahedron
  or the $B_3$ root lattice.
  The proof requires the solution in roots of unity
  of a $W(D_6)$-symmetric$\!$ polynomial$\!$ equation$\!$ with$\!$ $105\!$ monomials$\!$ (the$\!$ previous$\!$ record$\!$ was$\!$ $12\!$ monomials).
\end{abstract}

\maketitle

\section{Introduction}\label{S:introduction}

\subsection{Rational-angle line configurations}

Call an angle \defi{rational} if its degree measure is rational,
or equivalently if its radian measure is in $\Q \pi$.
Our main theorem classifies all sets $S$ of nonzero vectors in $\R^3$
such that the angle formed by each pair is rational.

Scaling a nonzero vector $\vv$ does not affect whether
the angles it forms with other vectors are rational,
so it is natural to consider the lines $\R \vv$.
In this paper, \defi{line} means line in $\R^3$ through $\mathbf{0}$,
and \defi{plane} is defined similarly.
A \defi{rational-angle line configuration} is a set of lines
such that each pair forms a rational angle.
Call two configurations \defi{equivalent}
if there exists an orthogonal transformation mapping one to the other.

\begin{example}
  Let $L \subset P$ be a line and plane.
  The set of lines in $P$ forming a rational angle with $L$
  together with the line perpendicular to $P$
  is a rational-angle line configuration.
  Call it a \defi{perpendicular configuration}.
  See the first image in Figure~\ref{F:three}.
\end{example}

\def\idsolid{

  \coordinate (p1) at (-0.50000,0.80902,-1.30902);
  \coordinate (p2) at (0.50000,0.80902,-1.30902);
  \coordinate (p3) at (0.00000,0.00000,-1.61803);
  \coordinate (p4) at (1.61803,0.00000,0.00000);
  \coordinate (p5) at (1.30902,-0.50000,-0.80902);
  \coordinate (p6) at (1.30902,0.50000,-0.80902);
  \coordinate (p7) at (0.50000,-0.80902,-1.30902);
  \coordinate (p8) at (0.80902,-1.30902,-0.50000);
  \coordinate (p9) at (-1.30902,0.50000,-0.80902);
  \coordinate (p10) at (-0.80902,1.30902,-0.50000);
  \coordinate (p11) at (0.50000,0.80902,1.30902);
  \coordinate (p12) at (1.30902,0.50000,0.80902);
  \coordinate (p13) at (0.80902,1.30902,0.50000);
  \coordinate (p14) at (-0.50000,-0.80902,-1.30902);
  \coordinate (p15) at (1.30902,-0.50000,0.80902);
  \coordinate (p16) at (0.00000,1.61803,0.00000);
  \coordinate (p17) at (-0.80902,1.30902,0.50000);
  \coordinate (p18) at (0.80902,1.30902,-0.50000);
  \coordinate (p19) at (-0.50000,0.80902,1.30902);
  \coordinate (p20) at (0.80902,-1.30902,0.50000);
  \coordinate (p21) at (-1.30902,-0.50000,-0.80902);

  \filldraw[fill=red, draw=black, thick]
  (p1) -- (p2) -- (p3) -- (p1)
  (p4) -- (p5) -- (p6) -- (p4)
  (p7) -- (p5) -- (p8) -- (p7)
  (p1) -- (p9) -- (p10) -- (p1)
  (p11) -- (p12) -- (p13) -- (p11)
  (p7) -- (p3) -- (p14) -- (p7)
  (p15) -- (p4) -- (p12) -- (p15)
  (p16) -- (p10) -- (p17) -- (p16)
  (p2) -- (p6) -- (p18) -- (p2)
  (p16) -- (p13) -- (p18) -- (p16);

  \filldraw[fill=yellow, draw=black, thick]
  (p19) -- (p11) -- (p13) -- (p16) -- (p17) -- (p19)
  (p20) -- (p8) -- (p5) -- (p4) -- (p15) -- (p20)
  (p21) -- (p9) -- (p1) -- (p3) -- (p14) -- (p21)
  (p1) -- (p10) -- (p16) -- (p18) -- (p2) -- (p1)
  (p7) -- (p3) -- (p2) -- (p6) -- (p5) -- (p7)
  (p6) -- (p18) -- (p13) -- (p12) -- (p4) -- (p6);

  \filldraw [black]
  (p1) circle [radius=2pt]
  (p2) circle [radius=2pt]
  (p3) circle [radius=2pt]
  (p4) circle [radius=2pt]
  (p5) circle [radius=2pt]
  (p6) circle [radius=2pt]
  (p7) circle [radius=2pt]
  (p8) circle [radius=2pt]
  (p9) circle [radius=2pt]
  (p10) circle [radius=2pt]
  (p11) circle [radius=2pt]
  (p12) circle [radius=2pt]
  (p13) circle [radius=2pt]
  (p14) circle [radius=2pt]
  (p15) circle [radius=2pt]
  (p16) circle [radius=2pt]
  (p17) circle [radius=2pt]
  (p18) circle [radius=2pt]
  (p19) circle [radius=2pt]
  (p20) circle [radius=2pt]
  (p21) circle [radius=2pt]
  ;
}

\def\b3{

  \filldraw[fill=green!20!white, draw=black, thick]
  (1,1,1) -- (1,-1,1) -- (-1,-1,1) -- (-1,1,1) -- (1,1,1);

  \filldraw[fill=blue!20!white, draw=black, thick]
  (1,-1,1) -- (1,-1,-1) -- (-1,-1,-1) -- (-1,-1,1) -- (1,-1,1);

  \filldraw[fill=red!20!white, draw=black, thick]
  (-1,1,1) -- (-1,-1,1) -- (-1,-1,-1) -- (-1,1,-1) -- (-1,1,1);

  \filldraw [violet]
  (0,0,1) circle [radius=2pt]
  (0,-1,0) circle [radius=2pt]
  (-1,0,0) circle [radius=2pt]
  ;

  \filldraw [black]
  (1,0,1) circle [radius=2pt]
  (-1,0,1) circle [radius=2pt]
  (0,1,1) circle [radius=2pt]
  (0,-1,1) circle [radius=2pt]
  (0,-1,-1) circle [radius=2pt]
  (1,-1,0) circle [radius=2pt]
  (-1,-1,0) circle [radius=2pt]
  (-1,1,0) circle [radius=2pt]
  (-1,0,-1) circle [radius=2pt]
  ;
}

\usetikzlibrary{perspective}

\begin{figure}[h]
  \phantom{mmm}
  \begin{tikzpicture}[scale = 1]
    \pgfmathsetmacro{\num}{17}
    \draw[red, very thick] (0,-1.25) -- (0,1.25);
    \foreach \n in {1, ..., \num} \draw ({-2*cos(\n*180/\num)*cos(15)},{-2*sin(\n*180/\num)*sin(15)}) --({2*cos(\n*180/\num)*cos(15)},{2*sin(\n*180/\num)*sin(15)});
    \draw[blue,thick] (-2,0) -- (2,0);
    \node[blue, left] at (-2,0) {$L$};
  \end{tikzpicture}
  \phantom{mmm}
  \begin{tikzpicture}[3d view = {-45}{35}, scale=0.8]
    \idsolid
  \end{tikzpicture}
  \phantom{mmmm}
  \begin{tikzpicture}[3d view , scale=1]
    \b3
  \end{tikzpicture}
  \phantom{mmmmmmm}
  \caption{A perpendicular configuration, an icosidodecahedron, and the $B_3$ root system.}
  \label{F:three}
\end{figure}
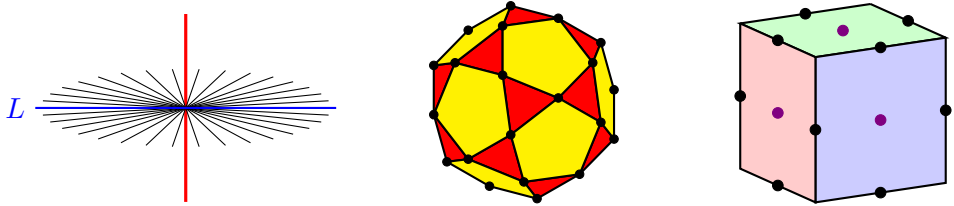

Any subset of a rational-angle line configuration is another,
so it suffices to classify \defi{maximal} rational-angle line configurations,
those not contained in a strictly larger one.
For $n<4$, describing the rational-angle configurations of $n$ lines is trivial
since there are no \emph{equations} that the angles between them must satisfy,
only the obvious \emph{inequalities}.

\begin{theorem}
  \label{T:main}
  The maximal rational-angle line configurations,
  up to equivalence,
  fall into finitely many families and sporadic examples as enumerated in
  Table \ref{Table:number of configurations}.
  In particular, each rational-angle line configuration not contained in a
  perpendicular configuration has at most $15$ lines.
\end{theorem}

\begin{table}
  {
    \begin{tabular}{c|l}
      $n$        & number of maximal rational-angle $n$-line configurations         \\ \hline
      $\aleph_0$ & 1                                                                \\
      15         & 1                                                                \\
      9          & 1                                                                \\
      8          & 5                                                                \\
      6          & 22, plus 5 one-parameter families                                \\
      5          & 29, plus 2 one-parameter families                                \\
      4          & 228, plus 10 one-parameter families and 2 two-parameter families \\
      3          & 1 three-parameter family
    \end{tabular}
    \bigskip
  }
  \caption{The number of maximal rational-angle line configurations with $n$ lines, up to equivalence.  For each $n$ not shown, there are none.  For the definition of family, see Definition~\ref{D:family of line configurations}.
    For a complete description of the families, see Section~\ref{sec:tables}.}
  \label{Table:number of configurations}
\end{table}

Here are geometric descriptions of the three largest configurations:

\begin{example}
  The $\aleph_0$-line configuration is the perpendicular configuration.
\end{example}

\begin{example}
  The $15$-line configuration consists of the lines connecting
  an icosidodecahedron's center to each of its $30$ vertices.
  (The vertices of an icosidodecahedron
  are the midpoints of the edges of a regular icosahedron,
  or equivalently, the midpoints of the edges of a regular dodecahedron;
  see the second image in Figure~\ref{F:three}.)
  The angles formed are all the multiples of $\pi/2$, $\pi/3$, $\pi/5$
  in $(0,\pi)$.
\end{example}

\begin{example}
  The $9$-line configuration consists of the lines in the directions
  of the $18$ roots of the $B_3$ root lattice
  (or equivalently, the $C_3$ root lattice, since the lengths are irrelevant).
  The angles formed are all the multiples of $\pi/3$ and $\pi/4$ in $(0,\pi)$.
  See the third image in Figure~\ref{F:three}.
\end{example}

Some additional examples are described in Section~\ref{sec:more lines}.

\begin{remark}
  \label{R:equivalent problems}
  The following problems are equivalent:
  \begin{enumerate}[\upshape (a)]
    \item classifying sets of nonzero vectors in $\R^3$ forming rational angles;
    \item classifying rational-angle line configurations;
    \item classifying \defi{rational-angle plane configurations}, i.e.,
          sets of planes such that each pair forms a rational angle
          (proof: take the perpendicular subspaces);
    \item classifying \defi{spherical codes} with distances in $\Q\pi$,
          i.e., subsets of the unit sphere such that
          the spherical distance between any two points lies in $\Q\pi$
          (proof: intersect each line with the sphere); and
    \item classifying convex polyhedra such that every two extended faces either form a rational angle or are parallel (for each rational-angle plane configuration $\mathscr{P}$ whose normal vectors span $\R^3$, choose closed half-spaces bounded by one or two planes parallel to each plane in $\mathscr{P}$, and consider their intersection, if bounded).
  \end{enumerate}
  Therefore Theorem~\ref{T:main} solves all of them.
\end{remark}

\begin{remark}
  There exist polyhedra with rational dihedral angles
  having two extended faces meeting at an irrational angle outside the polyhedron.
  These we do not classify in general.
\end{remark}

\subsection{Tetrahedra}

Call a tetrahedron \defi{rational} if all six of its dihedral angles
are rational.
Rational tetrahedra have Dehn invariant $0$,
or equivalently are scissors-congruent to a cube \cites{Dehn1901,Sydler1965},
and as such are candidates for
tetrahedra that can tile $\R^3$ \cite{Debrunner1980},
the study of which dates back to Aristotle \cite{Senechal1981}.
Conway and Jones in 1976 called attention to
the problem of classifying rational tetrahedra \cite{Conway-Jones1976}*{p.~239}.
We solve the problem in Theorem~\ref{T:tetrahedra} below.

A plane configuration is in \defi{general position}
if any three planes intersect in a point,
or equivalently, if in the corresponding line configuration,
no three lines are contained in any plane.
Rational tetrahedra up to similarity are in bijection
with rational-angle $4$-plane configurations in general position
up to equivalence:
given a tetrahedron, take the plane through $\mathbf{0}$ parallel to each face.
Because of this and Remark~\ref{R:equivalent problems},
Theorem~\ref{T:main} contains the classification of
rational tetrahedra.

Given a tetrahedron with vertices labeled $1,2,3,4$,
let $\alpha_{ij}$ be the dihedral angle along the edge joining vertices $i$ and $j$,
and list dihedral angles in the order
$(\alpha_{12},\alpha_{34},\alpha_{13},\alpha_{24},\alpha_{14},\alpha_{23})$
so as to pair each edge with the opposite edge.

\begin{theorem}
  \label{T:tetrahedra}
  The rational tetrahedra are those with dihedral angles
  \begin{gather*}
    \qquad\qquad\left( \pi/2, \; \pi/2, \; \pi - 2 x, \; \pi/3, \; x, \; x \right) \qquad \qquad \quad
    \textup{for } \pi/6 < x < \pi/2, \\
    \left( 5\pi/6 - x, \; \pi/6 + x, \; 2\pi/3 - x, \; 2\pi/3 - x, \; x, \; x \right) \qquad \textup{for } \pi/6 < x \leq \pi/3,
  \end{gather*}
  and the $59$ sporadic tetrahedra listed in Table~\ref{Table:tetrahedra}.
  $($Here, $x \in \Q\pi$ is assumed.$)$
\end{theorem}

\begin{remark}
  The first family in Theorem~\ref{T:tetrahedra}
  was discovered in 1895 \cite{Hill1895}*{Art.~4};
  see also \cite{Hadwiger1951} for a generalization to higher dimension
  and \cite{Maehara-Martini2018}*{\S2} for an elegant calculation of its angles.
  The second family in Theorem~\ref{T:tetrahedra} appears to be new.

  Of the $59$ sporadic rational tetrahedra, $15$
  (the tetrahedra $H_2(\pi/4)$, $T_0$--$T_7$, $T_{13}$, $T_{16}$--$T_{18}$, $T_{21}$, $T_{23}$ in \cite{Boltianskii1978}*{pp.~170--173})
  were discovered between 1895 and 1974
  \citelist{ \cite{Hill1895} \cite{Coxeter1948}*{p.~192} \cite{Sydler1956} \cite{Goldberg1958} \cite{Lenhard1962} \cite{Goldberg1974new} },
  and the other $44$ appear to be new.
\end{remark}

\begin{remark}
  We can ``explain'' almost all of the sporadic rational tetrahedra:
  Under the action of the Regge group $\Regge$
  (see Section~\ref{section:Regge symmetry}),
  $56$ of the $59$ are equivalent to a tetrahedron coming from a 4-line subconfiguration
  of the 15- or 9-line configuration.
  The remaining three are in the $\Regge$-orbit of the tetrahedron with dihedral angles
  $ \left(\pi/7, \; 3\pi/7, \;  \pi/3, \; \pi/3, \;  4\pi/7, \; 4\pi/7 \right)$.
\end{remark}

\subsection{Strategy of proof}

Geometry reduces the problem of determining
rational-angle $4$-line configurations
to solving a polynomial equation whose variables
are constrained to lie in the set $\mu$ of all roots of unity.
There are two known methods for solving equations in roots of unity;
one is practical for equations in up to $12$ monomials,
and the other is practical for equations in up to $3$ variables, roughly.
The complexity of each algorithm grows faster than exponentially.

What distinguishes our equation is that it has $105$ monomials in $6$ variables!
To solve it, we need the key idea, 
never before used to solve equations in roots of unity in characteristic~$0$, 
of building upon work of Dvornicich and Zannier \cite{Dvornicich-Zannier2002}
by working first 
in the quotient $\Z[\mu]/(2)$ of the subring $\Z[\mu] \subset \C$;
this makes the problem barely doable:
\begin{enumerate}[\upshape 1.]
  \item Reducing modulo $2$ yields a polynomial equation
        in $\Z[\mu]/(2)$ with only $12$ monomials!
  \item We adapt the first method above to
        parametrize all solutions in $\mu$ to such equations in $\Z[\mu]/(2)$.
        This restricts the possible $6$-tuples to lie in
        finitely many families, each parametrized by at most $3$ variables.
  \item Substituting each parametrization back into the original equation
        yields a polynomial equation (no longer mod $2$) in at most $3$ variables.
  \item We solve each of these equations using the second method above.
\end{enumerate}
Actually, we do not fully solve the equations as above,
but we do enough to constrain the roots of unity in sporadic solutions
to be of certain orders up to $840$;
then a large numerical computation, followed by an algebraic certification of results,
handles these ``small'' cases.
This yields a description of all $4$-line configurations,
in terms of 84696 parametrized families and sporadic examples of angle matrices
recording the pairwise angles between vectors along the lines.
These include the configurations corresponding to the tetrahedra in Theorem~\ref{T:tetrahedra}
but also many others in which at least three of the lines lie in a plane.
Finally, the $n$-line configurations for $n=5,6,\ldots,16$ in turn are determined
by finding all $n \times n$ matrices for which each $4 \times 4$ principal submatrix
belongs to one of the 84696 families; we employ an ``early abort'' strategy to
avoid having to analyze $84696^{\binom{16}{4}}$ cases.
The code for the various computations, written in
\cplusplus, Magma \cite{Magma}, SageMath \cite{SageMath}, and Singular \cite{Singular},
is available at \url{https://github.com/kedlaya/tetrahedra/}.

\begin{remark}
  Dvornicich, Veneziano, and Zannier \cite{Dvornicich-Veneziano-Zannier-preprint}
  study the rational angles
  formed by vectors in a lattice in $\R^2$.
  This leads to a problem of a different type,
  involving up to three variables constrained to be roots of unity,
  but also some variables constrained to be integers.
  Their analysis requires the determination of the rational points
  on some curves of genus $>1$.
\end{remark}

\section{Realizability of angle matrices}

\begin{definition}
  If $A = (a_{ij}) \in \M_n(\R)$
  and $I \subset \{1,\ldots,n\}$ with $|I|=m$,
  then $(a_{ij})_{i,j \in I} \in \M_m(\R)$
  is called an $m \times m$ \defi{principal submatrix} of $A$.
  Its determinant is called a \defi{principal minor} of $A$.

  Given nonzero $\vv,\ww \in \R^d$, let $\angle \vv \ww \in [0,\pi]$
  be the radian measure of the angle they form.
  Let $\Sphere^{d-1}$ be the unit sphere in $\R^d$; its elements are unit vectors.
  Let $\M_n(\R)^{\sym}_0$ be the set of symmetric $n \times n$ matrices
  with diagonal entries equal to $0$.
  Call $\Theta \in \M_n(\R)$ \defi{realizable in $\R^d$}
  if it is in the image of
  \begin{align*}
    (\Sphere^{d-1})^n    & \stackrel{\Angle}\To \M_n(\R)^{\sym}_0 \\
    (\vv_1,\ldots,\vv_n) & \longmapsto (\angle \vv_i \vv_j).
  \end{align*}
\end{definition}

\begin{proposition}
  \label{P:realizability of angle matrices}
  Suppose that $\Theta = (\theta_{ij}) \in \M_n(\R)^{\sym}_0$
  has entries in $[0,\pi]$.
  Let $C = (\cos \theta_{ij})$.
  Then $\Theta$ is realizable in $\R^d$ if and only if
  \begin{enumerate}[\phantom{mm}\upshape 1.]
    \item for every $m \le d$,
          each $m \times m$ principal submatrix of $C$ is positive semidefinite, and
    \item each $(d+1) \times (d+1)$ principal minor of $C$ equals $0$.
  \end{enumerate}
\end{proposition}

\begin{proof}
  See the proof of Lemma~2.1 of~\cite{Boroczky-Glazyrin2017}.
\end{proof}

\begin{corollary}
  \label{C:realizable}
  Let $\Theta \in \M_n(\R)$ for some $n \ge d+1$.
  Then $\Theta$ is realizable in $\R^d$
  if and only if every $(d+1) \times (d+1)$ principal submatrix of $\Theta$
  is realizable in $\R^d$.
\end{corollary}

\begin{remark}
  \label{R:1 and 2 minors}
  The $1 \times 1$ and $2 \times 2$ principal submatrices of $C$
  in Proposition~\ref{P:realizability of angle matrices}
  are automatically positive semidefinite,
  since $1$ and $1-\cos^2 \theta_{ij}$ are nonnegative.
\end{remark}

\begin{remark}
  \label{R:triangle inequalities}
  The nonnegative real numbers $\alpha, \beta, \gamma$
  are sides of a possibly degenerate spherical triangle
  if and only if
  $\alpha \le \beta + \gamma$,
  $\; \beta \le \gamma + \alpha$,
  $\; \gamma \le \alpha + \beta$, and
  $\alpha + \beta + \gamma \le 2 \pi$.
  Therefore, such angle inequalities give the condition
  for a $3 \times 3$ principal submatrix of $C$
  as in Proposition~\ref{P:realizability of angle matrices}
  to be positive semidefinite.
\end{remark}

Let $\calP_n \subset \M_n(\R)^{\sym}_0$
be the polytope defined by the $4 \binom{n}{3}$ inequalities,
four as in Remark~\ref{R:triangle inequalities}
from each of the $3 \times 3$ principal submatrices of $\Theta$.
Let $\calH_n \subset \M_n(\R)^{\sym}_0$
be the analytic subvariety defined by the vanishing of
the determinants of the $4 \times 4$ Gram matrices $(\cos \theta_{ij})_{i,j \in I}$,
one for each $4$-element subset $I \subset \{1,\ldots,n\}$.

\begin{corollary}
  \label{C:realizable locus}
  The set of $\Theta \in \M_n(\R)$ realizable in $\R^3$ is $\calP_n \intersect \calH_n$.
\end{corollary}

\begin{proof}
  Combine the $d=3$ case of
  Proposition~\ref{P:realizability of angle matrices}
  with Remarks \ref{R:1 and 2 minors} and~\ref{R:triangle inequalities}.
\end{proof}

\begin{definition}
  \label{D:family of angle matrices}
  \hfill
  \begin{enumerate}[\upshape (i)]
    \item
          A \defi{family of $\R^3$-realizable $n \times n$ rational-angle matrices}
          is a polytope $Q$ contained in $\calP_n \intersect \calH_n$
          such that
          \begin{itemize}
            \item the vertices of $Q$ are matrices with entries in $\Q\pi$.
            \item some element of $Q$ has no off-diagonal angles equal to $0$ or $\pi$; and
            \item $Q$ is not strictly contained in another polytope satisfying these   conditions.
          \end{itemize}
    \item
          The \defi{number of parameters} of the family is the dimension of $Q$.
    \item
          Call $Q$ \defi{maximal} if there is no family $Q'$ of
          $\R^3$-realizable $(n+1) \times (n+1)$ rational-angle matrices
          such that $Q$ equals the set of upper left principal submatrices of the matrices in $Q'$.
  \end{enumerate}
\end{definition}

\begin{definition}
  \label{D:family of line configurations}
  An \defi{$r$-parameter family of rational-angle line configurations}
  is the set of line configurations
  represented by all matrices with entries in $\Q\pi$
  belonging to a particular polytope $Q$ as described in Definition~\ref{D:family of angle matrices}.
\end{definition}

\section{Subvarieties of algebraic tori}
\label{section:algebraic tori}

Identify $\M_4(\R)^{\sym}_0$ with $\R^6$ via
$\Theta \mapsto (\theta_{12},\theta_{34},\theta_{13},\theta_{24},\theta_{14},\theta_{23})$.
Let $\calP = \calP_4 \subset [0,\pi]^6$.
Let $\calH = \calH_4 \subset \R^6$; it is the analytic hypersurface
\begin{align}
  \label{eq:gram det 4x4}
   & \det \, (\cos \theta_{ij})_{1 \leq i,j \leq 4}
  =0.
\end{align}
Expanding \eqref{eq:gram det 4x4} and substituting
$\cos \theta = (e^{\cplxi \theta} + e^{-\cplxi \theta})/2$
yields the six-variable equation
\begin{equation}
  \label{E:105} 
  -20
  + 4 \sum z_{12}^{\pm 1} z_{13}^{\pm 1} z_{23}^{\pm 1}
  - 2 \sum z_{12}^{\pm 2}
  - 2 \sum z_{12}^{\pm 1} z_{13}^{\pm 1} z_{24}^{\pm 1} z_{34}^{\pm 1}
  +   \sum z_{12}^{\pm 2} z_{34}^{\pm 2}
  = 0
\end{equation}
in which each sum ranges over the $S_4$-orbit of each monomial
and over all possible choices of signs.
The number of monomials is
$1 + 4 \cdot 2^3 + 6 \cdot 2^1 + 3 \cdot 2^4 + 3 \cdot 2^2 = 105$.

Let $Z$ be the subvariety of the algebraic torus $\G_m^6$ over $\Q$
defined by \eqref{E:105}.
Let $\exp \colon \R^6 \to \G_m^6(\C)$
be the map applying $\theta \mapsto e^{\cplxi \theta}$ to each coordinate,
so $\calH = \exp^{-1}(Z(\C))$.
The monomials appearing in \eqref{E:105}
generate an index-$8$ subgroup $\Lambda$
of the group of all Laurent monomials in the $z_{ij}$;
let $T$ be the torus whose coordinate ring is their span.
Thus there is an isogeny $\tau \colon \G_m^6 \to T$
and a closed subvariety $Y \subset T$
such that $Z = \tau^{-1} Y$.
The kernel of $\tau$ is the elementary abelian group of order~$8$
consisting of $(z_{ij}) \in \{\pm 1\}^6$ such that
$z_{ij} z_{jk} z_{ik} = 1$ for all $i<j<k$.
To summarize, we have a cartesian diagram of spaces
\begin{equation}
  \label{E:spaces}
  \begin{split}
    \xymatrix@R=5pt@C=40pt{
    & \calH \ar[r] \ar@{^{(}->}[d] & Z(\C) \ar@{->>}[r] \ar@{^{(}->}[d] & Y(\C) \ar@{^{(}->}[d] \\
    (\Sphere^2)^4 \ar[r]^-{\Angle} & \M_4(\R)^{\sym}_0 \isom \R^6 \ar[r]^-{\exp} & \G_m^6(\C) \ar@{->>}[r]^{\tau} & T(\C)
    }
  \end{split}
\end{equation}

For an abelian group $G$, let $G_{\tors}$ be its torsion subgroup.
The following problems are equivalent:
\begin{enumerate}[\phantom{mmm}1.]
  \item Determine all rational-angle $4$-line configurations.
  \item Determine $\calP \intersect \calH \intersect (\Q \pi)^6$.
        (Here we use the $n=4$ case of Corollary~\ref{C:realizable locus}.)
  \item Determine $Z(\C) \intersect \mu^6$.  (We have $\theta \in \Q\pi$ if and only if $e^{\cplxi \theta} \in \mu$.  We dropped the inequalities defining $\calP$, but these are easy to impose at the end of the computation.)
  \item Determine $Y(\C) \intersect T(\C)_{\tors}$.
\end{enumerate}

To solve 1, we will solve 3, but we will also use that $Z = \tau^{-1} Y$
and that $Y$ has additional symmetry described in the next section.

\section{Regge symmetry}
\label{section:Regge symmetry}

The signed permutation group $S_n^{\pm} \colonequals S_n \ltimes \{\pm 1\}^n$
acts on $(\Sphere^{d-1})^n$ by permuting and negating the $n$ vectors.
Similarly, $S_n$ acts on $\M_n(\R)^{\sym}_0$
by simultaneously permuting rows and columns,
and the $i$th generator of $\{\pm 1\}^n$ acts affine-linearly
by applying $x \mapsto \pi-x$
to each entry of the $i$th row and $i$th column except the $(i,i)$ entry.
The element $(-1,\ldots,-1)$ acts trivially on $\M_n(\R)^{\sym}_0$.

Now let $n=4$ and $d=3$.
The $S_4^\pm$-action on $\M_4(\R)^{\sym}_0$
is compatible with algebraic actions
of $S_4^\pm$ on $\G_m^6$ (not fixing $1$) and $T$ (fixing $1$)
such that the maps in the bottom row of \eqref{E:spaces}
are $S_4^\pm$-equivariant.

The $S_4^{\pm}$-action on $\M_4(\R)^{\sym}_0$ preserves $\calH$ and $\calP$.
Surprisingly, there is a larger group that preserves $\calH$ and $\calP$,
coming from exotic symmetries of the space of labeled tetrahedra,
as we will explain.

Fix an unordered partition of $\{1,2,3,4\}$ into pairs,
say $\{\{1,2\},\{3,4\}\}$, which we abbreviate as $12,\!34$.
Following \cite{Regge1959},
let $r=r_{12,34}$ be the linear operator on $\M_4(\R)^{\sym}_0 \isom \R^6$
sending $(x_{ij})$ to $(x_{ij}')$
where $x_{12}' \colonequals x_{12}$, $x_{34}' \colonequals x_{34}$,
and $x_{ij}' \colonequals s-x_{ij}'$ for all other $i<j$,
where $s \colonequals (x_{13} + x_{24} + x_{14} + x_{23})/2$.

Let $\Tetrahedron \subset \R^3$ be a \defi{labeled} tetrahedron;
labeled means that the vertices are numbered $1,2,3,4$.
For each $i \ne j$, let $e_{ij}$ be the edge connecting vertices $i$ and $j$,
let $\ell_{ij}$ be the length of $e_{ij}$,
and let $\alpha_{ij}$ be the dihedral angle along $e_{ij}$.
Define $L_\Tetrahedron = (\ell_{ij})$ and $A_\Tetrahedron = (\alpha_{ij})$;
both are in $\M_4(\R)^{\sym}_0$.

\begin{theorem}[Ponzano and Regge]
  \label{T:Ponzano}
  For each labeled tetrahedron $\Tetrahedron$,
  there exists a labeled tetrahedron $\Tetrahedron'$, unique up to congruence,
  such that $L_{\Tetrahedron'} = r L_\Tetrahedron$ and $A_{\Tetrahedron'} = r A_\Tetrahedron$.
  Moreover, $\Tetrahedron$ and $\Tetrahedron'$ are scissors-congruent.
\end{theorem}

\begin{proof}
  The first statement was proved in \cite{Ponzano-Regge1968}*{Appendices B and~D}
  by a brute force calculation.
  Geometric proofs have recently been discovered
  \cites{Akopyan-Izmestiev2019,Rudenko-preprint},
  but they are not simple.
  The scissors congruence was first observed in \cite{Roberts1999}*{Corollary~10}.
\end{proof}

\begin{remark}
  The same theorem holds in spherical and hyperbolic geometry:
  see \cites{Mohanty2003,Taylor-Woodward2005,Akopyan-Izmestiev2019,Rudenko-preprint}.
\end{remark}

\begin{definition}
  Call the operator $r=r_{12,34}$ and its analogues $r_{13,24}$ and $r_{14,23}$
  \defi{Regge operators}.
  Together with $S_4$,
  they generate a subgroup $\Regge \subset \GL(\M_4(\R)^{\sym}_0) \isom \GL_6(\R)$;
  in fact, $r$ and $S_4$ already generate $\Regge$ since the other Regge operators
  are $S_4$-conjugates of $r$.
  The group $\Regge$ is isomorphic to $S_4 \times S_3$ \cite{Regge1959},
  but the isomorphism sends the original $S_4$
  to the graph of a surjection $S_4 \to S_3$,
  not a normal subgroup, let alone a direct factor.
  Let $\Regge^\pm$ be the subgroup of the affine linear group
  of $\M_4(\R)^{\sym}_0$ generated by the image of $S_4^\pm$
  and the Regge operators.
  Then $|\Regge| = 2^4 3^2$ and $|\Regge^\pm| = 2^7 3^2$.
\end{definition}

Identify the $z_{ij}$ with the standard basis of $\Z^6$,
but scale the Euclidean norm so that $\langle z_{ij},z_{ij} \rangle =1/2$.
Then $\Lambda$ is a lattice.
For each $c \in \Z$, let $\Lambda_c \subset \Lambda$ be the set of monomials
in~\eqref{E:105} with coefficient $c$.
Checking inner products shows that $\Lambda_{-2}$ is a copy of the $D_6$ root system!
Let $W(D_6)$ be the Weyl group, which we view as acting on the right on $\Lambda$,
so that it acts on the left on $T$.
For each $c \in \{-20,4,-2,1\}$, the set $\Lambda_c$ is a $W(D_6)$-orbit,
so $W(D_6)$ preserves $Y$.

The $S_4^\pm$-action on $T$ preserves the norm on $\Lambda$,
so it factors through $W(D_6)$.
A brief calculation shows that the action of $r$ on $\M_4(\C)^{\sym}_0$
corresponds to a linear action on $\Q^6 \isom \Lambda \tensor \Q$
that preserves $\Lambda_{-2}$ and hence is in $W(D_6)$,
so the homomorphism $S_4^\pm \to W(D_6)$
extends to $\Regge^{\pm} \to W(D_6)$.
Since \eqref{E:spaces} is cartesian, $\Regge^{\pm}$ preserves $\calH$.

In summary, we have a two-row cartesian diagram of spaces
and a sequence of homomorphisms of groups,
each acting on the spaces above it,
compatibly with respect to the homomorphisms:
\begin{equation}
  \label{E:actions}
  \begin{split}
    \xymatrix@R=5pt@C=40pt{
    & \calH \ar[r] \ar@{^{(}->}[d] & Z(\C) \ar@{->>}[r] \ar@{^{(}->}[d] & Y(\C) \ar@{^{(}->}[d] \\
    (\Sphere^2)^4 \ar[r]^-{\Angle} & \M_4(\R)^{\sym}_0 \isom \R^6 \ar[r]^-{\exp} & \G_m^6(\C) \ar@{->>}[r]^{\tau} & T(\C) \\
    S_4^\pm \ar[r] & \Regge^\pm \ar@{^{(}->}[rr] && W(D_6).
    }
  \end{split}
\end{equation}
Finally, $S_4^\pm$ preserves $\calP$,
and direct calculation shows that $r$ does too,
so $\Regge^\pm$ preserves $\calP$.


\section{Cyclotomic relations}
\label{sec:cyclotomic relations}

Recall from the end of Section~\ref{section:algebraic tori}
that we need to find the torsion points on a hypersurface $Z$ in a torus $\G_m^6$;
this amounts to solving \eqref{E:105} in roots of unity.
Prior to our work, there were two general approaches to solving such problems:
\begin{itemize}
  \item Classify integer relations involving few roots of unity (this section).
  \item Use the Galois theory of cyclotomic fields and induction on the dimension (Section~\ref{sec:torsion closures}).
\end{itemize}
But, crucially, we develop also a new method in Section~\ref{sec:mod 2 relations}
involving cyclotomic relations modulo $2$.
We need all three methods
to find the torsion points on our particular variety $Z$;
see Section~\ref{sec:solutions to Gram equation}.

The classification of additive relations among roots of unity
grows out of work of Gordan~\cite{Gordan1877},
de~Bruijn~\cite{deBruijn1953}, R\'edei~\cites{Redei1959, Redei1960}, and Schoenberg~\cite{Schoenberg1964}.
Relations among $n$ roots of unity have been classified for
$n \leq 7$ by Mann~\cite{Mann1965},
$n\leq 8$ by W{\l}odarski~\cite{Wlodarski1969},
$n\leq 9$ by Conway and Jones~\cite{Conway-Jones1976}*{Theorem~6},
and $n\leq 12$ by Poonen and Rubinstein~\cite{Poonen-Rubinstein1998}*{Theorem~3.1}.
The last of these has the following consequence.

\begin{theorem} \label{T:sums of 6 cosines}
  Let $x_1,\dots,x_n \in \QQ$ be a sequence with $n \leq 6$ such that
  $\sum_{i=1}^n \cos (2 \pi x_i) = 0$,
  but no nonempty proper subsequence has the same property.
  Then $x_1,\dots,x_n$ can be obtained from
  a sequence in Table~\ref{table:sums of cosines}
  by some combination of permutation of terms, individual negation,
  individual addition of integers, and simultaneous addition of $1/2$.
\end{theorem}

\begin{proof}
  Combine Theorem~3.1, Lemma~4.1, and Lemma~4.2 of \cite{Poonen-Rubinstein1998}.
\end{proof}

\renewcommand{\arraystretch}{1.25}
\begin{table}[ht]
  \begin{tabular}{c|c|c}
    Length & Type                   & Values                                                                                        \\
    \hline
    $1$    & n/a                    & $\frac{1}{4}$                                                                                 \\
    \hline
    $2^*$  & $2 R_2$                & $x + (0, h)$                                                                                  \\
    \hline
    $3$    & $(R_5 : R_3)$          & $\frac{1}{3}+h, \frac{1}{5}, \frac{2}{5}$                                                     \\
    \hline
    $3^*$  & $2R_3$                 & $x + (0, \frac{1}{3}, \frac{2}{3}) $                                                          \\
    \hline
    $4$    & $(R_5 : 3R_3)$         & $\frac{1}{3}, \frac{1}{3} + \frac{1}{5}, \frac{2}{3} + \frac{1}{5}, \frac{2}{5}+h$            \\
           &                        & $\frac{1}{3}, \frac{1}{3} + \frac{2}{5}, \frac{2}{3} + \frac{2}{5}, \frac{1}{5}+h $           \\
           & $(R_7 : R_3)$          & $\frac{1}{3}+h,\frac{1}{7}, \frac{2}{7}, \frac{3}{7}$                                         \\
    \hline
    $5$    & $(R_7 : 3R_3)$         & $\frac{1}{3},\frac{1}{3}+\frac{1}{7}, \frac{2}{3}+\frac{1}{7}, \frac{2}{7}+h, \frac{3}{7}+h$  \\
           &                        & $\frac{1}{3}, \frac{1}{3}+\frac{2}{7}, \frac{2}{3}+\frac{2}{7}, \frac{1}{7}+h, \frac{3}{7}+h$ \\
           &                        & $\frac{1}{3}, \frac{1}{3}+\frac{3}{7}, \frac{2}{3}+\frac{3}{7}, \frac{1}{7}+h, \frac{2}{7}+h$ \\
           & $ (R_7: R_5)$          & $\frac{1}{5}+h, \frac{2}{5}+h,
      \frac{1}{7}, \frac{2}{7}, \frac{3}{7}$                                                                                          \\
    \hline
    $5^*$  & $2R_5$                 & $x + (0,\frac{1}{5}, \frac{2}{5}, \frac{3}{5}, \frac{4}{5})$                                  \\
    \hline
    $6$    & $(R_7 : 5R_3)$         &
    $\frac{1}{3},
      \frac{1}{3} + \frac{1}{7},
      \frac{1}{3} + \frac{2}{7},
      \frac{2}{3} + \frac{1}{7},
      \frac{2}{3} + \frac{2}{7},
      \frac{3}{7} + h$
    \\
           &                        & $\frac{1}{3},
      \frac{1}{3} + \frac{1}{7},
      \frac{1}{3} + \frac{3}{7},
      \frac{2}{3} + \frac{1}{7},
      \frac{2}{3} + \frac{3}{7},
      \frac{2}{7}+h$                                                                                                                  \\
           &                        & $\frac{1}{3},
      \frac{1}{3} + \frac{2}{7},
      \frac{1}{3} + \frac{3}{7},
      \frac{2}{3} + \frac{2}{7},
      \frac{2}{3} + \frac{3}{7},
      \frac{1}{7}+h$                                                                                                                  \\
           & $ (R_7 : R_5, 2R_3)$   &
    $\frac{1}{5},
      \frac{2}{5},
      \frac{1}{3} + \frac{1}{7},
      \frac{2}{3} + \frac{1}{7},
      \frac{2}{7}+h, \frac{3}{7}+h$                                                                                                   \\
           &                        & $\frac{1}{5},
      \frac{2}{5},
      \frac{1}{3} + \frac{2}{7},
      \frac{2}{3} + \frac{2}{7},
      \frac{1}{7}+h, \frac{3}{7}+h$                                                                                                   \\
           &                        & $\frac{1}{5},
      \frac{2}{5},
      \frac{1}{3} + \frac{3}{7},
      \frac{2}{3} + \frac{3}{7},
      \frac{1}{7}+h, \frac{2}{7}+h$                                                                                                   \\
           & $(R_7 : (R_5 : 2R_3))$ &
    $\frac{2}{5}+h, \frac{1}{3} + \frac{1}{5}, \frac{2}{3} + \frac{1}{5}, \frac{1}{7}, \frac{2}{7}, \frac{3}{7}$                    \\
           &                        &
    $\frac{1}{5}+h, \frac{1}{3} + \frac{2}{5}, \frac{1}{3} + \frac{2}{5},
      \frac{1}{7}, \frac{2}{7}, \frac{3}{7}$                                                                                          \\
           & $(R_{11}: R_3)$        & $\frac{1}{3}+h, \frac{1}{11}, \frac{2}{11},\frac{3}{11},\frac{4}{11}, \frac{5}{11}$           \\
    \hline
    $6^*$  & $2(R_5 : R_3)$         & $x + (\frac{1}{3}+h, \frac{2}{3} +h, \frac{1}{5}, \frac{2}{5}, \frac{3}{5}, \frac{4}{5})$     \\
    \hline
  \end{tabular}
  \smallskip

  \caption{Indecomposable additive relations among at most 6 cosines of rational multiples of $2\pi$, up to transformations listed in Theorem~\ref{T:sums of 6 cosines}. The symbol $h$ stands for $\frac{1}{2}$.
    A length of $n^*$ indicates a shift by an auxiliary parameter $x \in \QQ$.
    The type is notated as per \cite{Poonen-Rubinstein1998}*{Table~3.1 and Table~3.2}.
    By Theorem~\ref{T:sums of 12 mod 2},
    this table (with one addition) also describes mod 2 cosine relations;
    in these, we may ignore shifts by $h$.}
  \label{table:sums of cosines}
\end{table}

\begin{remark}
  Building on these ideas,
  algorithms for finding the solutions of a polynomial equation in roots of unity
  have been described by
  Sarnak and Adams \cite{Sarnak-Adams1994};
  Filaseta, Granville, and Schinzel \cite{Filaseta-Granville-Schinzel2008};
  and Leroux \cite{Leroux2012}.
  But these algorithms scale exponentially
  in the number of variables and the number of monomials,
  so executing them on a polynomial with 105 monomials, as in \eqref{E:105}, is infeasible.
\end{remark}

\section{Mod 2 cyclotomic relations}
\label{sec:mod 2 relations}

Let $\overline{\mu}$ be the image of $\mu$ in $\Z[\mu]/(2)$,
so $\overline{\mu} \isom \mu/\{\pm 1\}$.
(We would lose too much information if instead we chose a prime $\pp$ above $2$
and worked in the residue field $\Z[\mu]/\pp \isom \Fbar_2$.)
For $n \ge 1$, let $\mu_n \colonequals \{z \in \mu : z^n=1\}$,
let $\overline{\mu}_n$ be the image of $\mu_n$ in $\Z[\mu]/(2)$,
and let $\zeta_n \colonequals e^{2\pi \cplxi/n} \in \mu_n$.

By a \emph{mod $2$ relation}, we mean a finite subset
$S \subset \overline{\mu}$ summing to $0$ in $\Z[\mu]/(2)$.
Call $S$ \defi{indecomposable}
if $S \ne \emptyset$ and $S$ is not
the disjoint union of two nonempty relations.
The \defi{weight} of $S$ is $w(S) \colonequals |S|$.
The \defi{level} $\ell(S)$ is the smallest $n \ge 1$ such that
$S \subset \overline{\mu}_n$.
Call relations $S$ and $S'$ \defi{equivalent}
if $S' = \lambda S$ for some $\lambda \in \overline{\mu}$.
Call $S$ \defi{minimal}
if $\ell(S) \le \ell(S')$ for all $S'$ equivalent to $S$.

The goal of this section is Theorem~\ref{T:sums of 12 mod 2},
the mod~2 analogue of Theorem~\ref{T:sums of 6 cosines}.
We follow the proof of \cite{Poonen-Rubinstein1998}*{Theorem~3.1}.
First we establish an analogue of \cite{Conway-Jones1976}*{Theorem~1}:

\begin{lemma}
  \label{L:minimal level}
  The level of any minimal indecomposable mod $2$ relation is odd and squarefree.
\end{lemma}

\begin{proof}
  Let $S$ be the relation.
  Let $N=\ell(S)$.
  Suppose that $p$ is a prime such that $p^2$ divides $N$.
  Since
  \[
    \ZZ[\zeta_N] = \ZZ[\zeta_{N/p}][T]/(T^p - \zeta_{N/p})
  \]
  and similarly after reduction mod 2, the intersection of $S$ with each coset of the group $\overline{\mu}_{N/p}$ is another relation.
  Since $S$ is indecomposable, it is contained in a single coset,
  so $S$ is equivalent to a relation of level dividing $N/p$, a contradiction.
  Thus $N$ is squarefree.
  If $N$ is even, then $\overline{\mu}_{N} = \overline{\mu}_{N/2}$,
  so $\ell(S) \ne N$.
\end{proof}

\begin{definition}
  For each odd prime $p$, let $R_p$ denote the set $\overline{\mu}_p$
  viewed as a weight~$p$ relation.
\end{definition}

Let $\directsum$ denote symmetric difference of sets.

\begin{definition}
  For relations
  $S, T_1,\dots,T_j$, let $(S: T_1,\dots,T_j)$ denote any relation of the form $S' \oplus T'_1 \oplus \cdots \oplus T'_j$,
  where $S', T'_1, \dots, T'_j$ are equivalent to $S,T_1,\dots,T_j$, respectively; $\# (S' \cap T'_i) = 1$ for all $i$;
  and $T'_i \cap T'_k = \emptyset$ whenever $i \ne k$.
  (The equivalence class of such a relation need not be determined by the equivalence classes of $S,T_1,\dots,T_j$.)
\end{definition}

The following is an analogue of \cite{Poonen-Rubinstein1998}*{Lemma~3.4},
and, by extension, of \cite{Conway-Jones1976}*{Theorem~5}. 
A direct analogue of the latter result, working modulo any prime,
can be found in \cite{Dvornicich-Zannier2002}.

\begin{lemma} \label{L:remove maximal prime from minimal relation}
  Let $S$ be a minimal indecomposable mod $2$ relation
  of level $p M$, where $p \nmid M$.
  If $S$ intersects some coset of $\overline{\mu}_M$
  in at most one element,
  then $S$ is of the form $(R_p: T_1,\dots,T_j)$,
  where $0 \le j<p$, each $T_i$ is nonempty with $\ell(T_i) \mid M$, and
  \begin{equation}
    \sum_{i=1}^j (w(T_i)-2) = w(S) - p.
  \end{equation}
\end{lemma}

\begin{proof}
  Reducing
  \[
    \ZZ[\zeta_{pM}] = \ZZ[\zeta_M][T]/(T^{p-1} + \cdots + T + 1),
  \]
  modulo $(2)$ shows that
  the intersections of $S$ with the cosets of $\overline{\mu}_M$ must have sums which are rotations of each other by powers of $\zeta_p$.
  No such intersection can be empty, or else each intersection would itself be a relation, equivalent to one of level dividing $M$, contradicting the hypotheses on $S$.  Therefore some intersection has one element.
  Then each of the other intersections is either itself a singleton set or the complement of a single root of unity in some mod $2$ relation.
\end{proof}

\begin{corollary} \label{C:no weight 4 relation}
  Each minimal indecomposable mod $2$ relation $S$ of weight at most $5$ is
  $R_3$ or $R_5$.
\end{corollary}

\begin{proof}
  Let $p$ be the largest prime dividing $\ell(S)$.
  In Lemma~\ref{L:remove maximal prime from minimal relation},
  $w(T_i) \ge 3$ for each $i$, so $w(S) \ge p$,
  with equality if and only if $S=R_p$.
  If $p=3$, then $w(S) \le \ell(S)=3 = p$;
  if $p \ge 5$, then $w(S) \le 5 \le p$.
  Thus the equality holds, with $p=3$ or $p=5$.
\end{proof}

\begin{lemma} \label{L:not all intersections of weight 2}
  Let $S$ be a minimal indecomposable mod $2$ relation of level $pM$
  with $p \nmid M$.
  Then the intersections of $S$ with the cosets of $\overline{\mu}_M$
  cannot all have exactly two elements.
\end{lemma}

\begin{proof}
  As in the proof of Lemma~\ref{L:remove maximal prime from minimal relation},
  $S \intersect \zeta_p^i \mu_m = \zeta_p^i U_i$
  for some two-element sets $U_0,\dots,U_{p-1}$ which all have equal sum.
  By Corollary~\ref{C:no weight 4 relation},
  each $U_i \directsum U_j$ must be empty, so $U_0 = \cdots = U_{p-1}$.
  Then $S$ is the union of two relations of type $R_p$,
  contradicting indecomposability.
\end{proof}

\begin{theorem} \label{T:classify mod 2 relations}
  For each $w \in \{3,\ldots,12\}$,
  the indecomposable mod $2$ relations of weight $w$ are precisely
  the mod $2$ reductions of the indecomposable relations of weight $w$
  listed in \cite{Poonen-Rubinstein1998}*{Table~3.1}.
\end{theorem}

(The statement of Theorem~\ref{T:classify mod 2 relations} must exclude $w=2$,
because $R_2$ reduces mod 2 to the empty relation.)

\begin{proof}[Proof of Theorem~\ref{T:classify mod 2 relations}]
  Let $S$ be a minimal indecomposable mod $2$ relation of level $N$ and weight $w$.
  By Lemma~\ref{L:minimal level}, we can write $N = p_1 \cdots p_s$
  where $2 < p_1 < \cdots < p_s$.
  By Lemma~\ref{L:remove maximal prime from minimal relation}, $p_s \le w \le 12$,
  so $p_s \leq 11$.

  \begin{itemize}
    \item Suppose that $p_s = 3$.
          Then Lemma~\ref{L:remove maximal prime from minimal relation} yields $S = R_3$.
    \item Suppose that $p_s = 5$.
          Each coset of $\overline{\mu}_3$ is itself a relation of weight $3$,
          so the intersection of $S$ with any such coset has at most two elements.
          By Lemma~\ref{L:not all intersections of weight 2},
          the intersections cannot all have exactly two elements,
          so Lemma~\ref{L:remove maximal prime from minimal relation}
          yields $S = (R_5: jR_3)$ for some $j \in \{0,\dots,4\}$.
    \item Suppose that $p_s = 7$.
          Then Lemma~\ref{L:remove maximal prime from minimal relation} implies
          that $S = (R_7:T_1, \dots, T_j)$ for some $j$ with
          $\sum_{i=1}^j (w(T_i)-2) \leq 5$.
          By the previous step, each $T_i$ must have one of the forms
          $R_3$, $R_5$, $(R_5:R_3)$, $(R_5:2R_3)$,
          which have $w(T_i)-2$ being $1$, $3$, $4$, $5$, respectively.
          By considering the partitions of $5$ into parts of these sizes,
          we obtain relations of the indicated forms.
    \item Suppose that $p_s = 11$.
          Then Lemma~\ref{L:remove maximal prime from minimal relation}
          yields $S = R_{11}$ or $S = (R_{11}:R_3)$.\qedhere
  \end{itemize}
\end{proof}

\begin{corollary} \label{C:classify mod 2 relations}
  Every mod $2$ relation of weight at most $12$
  is the reduction of a genuine cyclotomic relation
  $($i.e., a subset of $\mu$ summing to $0$ in $\Z[\mu]$$)$
  of the same weight.
\end{corollary}

\begin{proof}
  Reduce to the indecomposable case
  and apply Theorem~\ref{T:classify mod 2 relations}.
\end{proof}

To pass from mod $2$ cyclotomic relations to cosine relations, we argue as in
\cite{Poonen-Rubinstein1998}*{Lemma~4.1}. Keep in mind that the decomposition of a mod $2$ relation into indecomposable relations is not \emph{a priori} guaranteed to be unique.

\begin{lemma} \label{L:conjugation-stable mod 2 relations}
  Let $S$ be a mod $2$ relation with $w(S) \le 12$.
  Suppose that $S$ is stable under complex conjugation.
  \begin{enumerate}[\upshape (a)]
    \item \label{I:partition}
          There is a partition of $S$ in which each part is
          either a conjugation-stable indecomposable relation
          or the disjoint union of two conjugate indecomposable relations.
    \item
          If $S$ has even weight,
          then each conjugation-stable indecomposable relation in~\eqref{I:partition}
          has even weight.
  \end{enumerate}
\end{lemma}

\begin{proof}
  (a)  We use induction on $w(S)$.
  Let $T \subset S$ be any indecomposable relation.
  Let $T'$ be its conjugate.
  If $T = T'$ or $T \intersect T' = \emptyset$,
  remove $T \cup T'$ from $S$ and apply the inductive hypothesis.
  Otherwise, apply induction to $T \oplus T'$ and its complement in $S$.

  (b) A conjugation-stable relation has odd weight if and only if it contains 1.
\end{proof}

\begin{theorem} \label{T:sums of 12 mod 2}
  Let $x_1,\dots,x_n \in \QQ$ be a sequence with $n \leq 6$ such that $\sum_{j=1}^n 2 \cos (2 \pi x_j) \equiv 0 \pmod{2 \Z[\mu]}$, but no nonempty proper subsequence has the same property. Then either $n=1$ and $2x_1 \in \ZZ$, or
  the given sequence can be obtained from a sequence in Table~\ref{table:sums of cosines}
  by some combination of permutation of terms, individual negation, and \emph{individual} translation by half-integers.
\end{theorem}

\begin{proof}
  This follows by applying Theorem~\ref{T:classify mod 2 relations} and
  Lemma~\ref{L:conjugation-stable mod 2 relations}
  to the mod 2 cyclotomic relation
  coming from the sum
  $
    \sum_{j=1}^n \left( e^{2 \pi \cplxi x_j} + e^{-2 \pi \cplxi x_j} \right),
  $
  except in the case where this sum cancels completely mod 2.
  Given the indecomposability hypothesis on the original sequence
  (which implies in particular that the $x_j$ are distinct modulo $\frac12 \Z$),
  this happens only if $n=1$ and $2x_1 \in \ZZ$.
\end{proof}

\begin{remark} \label{R:weight-preserving bijection1}
  Theorem~\ref{T:classify mod 2 relations} implies that
  for each $w \in \{3,\ldots,12\}$,
  reduction modulo 2 defines a weight-preserving bijection
  between equivalence classes of weight~$w$ indecomposable cyclotomic relations
  and equivalence classes of 
  weight~$w$ indecomposable mod~$2$ cyclotomic relations,
  but this does not hold for all $w$.
  The cyclotomic polynomial $\Phi_{105}$ has two coefficients equal to $-2$;
  it yields a weight $35$ indecomposable cyclotomic relation
  reducing to a weight $31$ indecomposable mod $2$ relation;
  see \cite{Dvornicich-Zannier2002}*{p.~105}.
\end{remark}

\section{Torsion closures}
\label{sec:torsion closures}

Throughout this section, let $K$ be a subfield of $\CC$,
let $T$ be the torus $\GG_m^n = \Spec K[x_1^{\pm}, \ldots, x_n^{\pm}]$,
and let $X$ be a closed subscheme of $T$.
For $P \in T(K)$, let $t_P \colon T \to T$ be the translation-by-$P$ map.
For a positive integer $m$, let $[m] \colon T \to T$ be the $m$th power map,
and let $T[m] \subset T(\CC)$ be its kernel.

\begin{definition}
  A \defi{torsion coset} of $T$ is a translate of a subtorus of $T_{\C}$
  by a point in $\Union_{m \ge 1} T[m]$.
\end{definition}

\begin{definition}
  The \emph{torsion closure} of $X$ in $T$ is the Zariski closure of $X(\CC) \cap \mu^n$, viewed as a reduced $K$-subscheme of $X$.
\end{definition}

\begin{theorem}[Laurent]
  The torsion closure of $X_\C$ is a finite union of torsion cosets of $T$.
\end{theorem}

\begin{proof}
  This is a special case of \cite{Laurent1984}.
  It is also a special case of a conjecture of Lang \cite{Lang1983}*{Notes for Chapter~8} combining the Manin--Mumford and Mordell conjectures, which is itself now known in full generality \cite{McQuillan1995}. See also \cite{Hindry2006} for a survey.
\end{proof}

Since torsion cosets are definable over $\Q(\mu)$, 
the general problem of computing torsion closures
can be reduced to the case in which
$X$ is defined over $K = \Q(\zeta_N)$ for some $N$; 
see \cite{Aliev-Smyth2012}*{\S3.3}.

The key idea behind our algorithm for computing torsion closures
is that certain field automorphisms act on torsion points in the same way
as certain morphisms of varieties; for example, there is an automorphism of $\C$
that acts on odd-order roots of unity in the same way as the squaring morphism
$\G_m \to \G_m$.
This idea appears in the proof of the case $n=2$ of Laurent's theorem by Ihara, Serre, and Tate \cite{Lang1983}*{\S 8.6}, and in subsequent presentations by Ruppert \cite{Ruppert1993}, Beukers and Smyth \cite{Beukers-Smyth2000}, and Aliev and Smyth \cite{Aliev-Smyth2012}.

In writing $K=\Q(\zeta_N)$,
we may assume that $N = 2^e m$ with $e \ge 1$ and $m$ odd.
If $e = 1$, choose $\tau \in \Aut K$ such that $\tau(\zeta_m)=\zeta_m^2$.
If $e \ge 2$, choose $\sigma \in \Aut K$ such that $\sigma(\zeta_N)=-\zeta_N$.
Extend $\tau$ to a $\Q$-automorphism of $K[x_1^{\pm},\dots,x_n^{\pm}]$
acting trivially on the $x_i$.
Then $\Spec \tau$ is a $\Q$-endomorphism of $T$.
Similarly, define $\Spec \sigma \in \Aut_\Q T$.
Define the following finite sets of $\Q$-endomorphisms of $T$:
\begin{align*}
  S_1 & \colonequals \{ t_P : P \in T[2] \setminus \{1\} \}, \\
  S_2 & \colonequals \begin{cases}
    \emptyset,                                 & \textup{if $e=1$,}     \\
    \{ (\Spec \sigma) \circ t_P: P \in T[2]\}, & \textup{if $e \ge 2$,}
  \end{cases}          \\
  S_3 & \colonequals \begin{cases}
    \{ (\Spec \tau) \circ t_P \circ [2] : P \in T[2]\}, & \textup{if $e=1$,}     \\
    \emptyset,                                          & \textup{if $e \ge 2$.} \\
  \end{cases}
\end{align*}
Let $S = S_1 \union S_2 \union S_3$.
Intersections of subschemes below are always scheme-theoretic intersections.

\begin{lemma} \label{L:covering of torsion points}
  The torsion closure of $X$ is contained in $\Union_{f \in S} (X \cap f^{-1} X)$.
\end{lemma}

\begin{proof}
  Each $\alpha \in \Aut \C$ induces a coordinatewise map $T(\C) \to T(\C)$.
  Suppose that $w \in T(\C)$ and $\alpha(w)=z$.
  Then $\Spec \alpha|_K$ maps the image of $\Spec \C \stackrel{z}\to T$
  to the image of $\Spec \C \stackrel{w}\to T$, because the diagram
  \[
    \xymatrix{
      \Spec \C \ar[d]_z \ar^{\Spec \alpha}[r] & \Spec \C \ar[d]^w \\
      T \ar[r]^{\Spec \alpha|_K} & T
    }
  \]
  commutes (check on rings); in particular, if $w \in X(\C)$,
  then $z \in ((\Spec \alpha|_K)^{-1} X)(\C)$.

  Let $z = (z_1,\ldots,z_n) \in X(\C)$ be a torsion point.
  Write $K(z) \cap \mu = \langle \zeta_M \rangle$, so $N | M$.
  \begin{itemize}
    \item
          Suppose that $4 \nmid M$.
          Extend $\tau$ to $\alpha \in \Aut \C$ such that $\alpha(\zeta_M) = -\zeta_M^2$.
          Then $\alpha(z_i)=\pm z_i^2$ for each $i$,
          so $\alpha(z)  = (t_P \circ [2])(z)$ for some $P \in T[2]$.
          By the first paragraph of the proof,
          $(t_P \circ [2])(z) \in (\Spec \tau)^{-1} X)(\C)$,
          so $z \in (f^{-1} X)(\C)$ for some $f \in S_3$.
    \item
          Suppose that $4|M$.
          Let $\sigma' \in \Aut K$ be $1$ or $\sigma$,
          according to whether $M/N$ is even or not.
          Then $\sigma'$ extends to $\alpha \in \Aut \C$
          such that $\alpha(\zeta_M) = -\zeta_M$.
          Then $\alpha(z_i) = \pm z_i$ for all $i$,
          so $\alpha(z) = t_P(z)$ for some $P \in T[2]$.
          If $P=1$, then $\alpha$ fixes $z$ but not $K(z)$,
          so $\sigma' = \alpha|_K \ne 1$.
          By the first paragraph again, $z \in (f^{-1} X)(\C)$
          for some $f$ in $S_1$ or $S_2$,
          according to whether $\sigma'$ is $1$ or $\sigma$.\qedhere
  \end{itemize}
\end{proof}

Lemma~\ref{L:covering of torsion points} suggests the following
recursive algorithm for computing the torsion closure of $X$.

\begin{algorithm} \label{algorithm:torsion closure}
  Suppose that $K = \QQ(\zeta_N)$.
  Given a closed subscheme $X$ of $T$, return another closed subscheme of $T$ as follows.
  \begin{enumerate}[\upshape 1.]
    \item
          If $X \subseteq f^{-1}X$ for some $f \in S$,
          then choose one such $f$ (using any deterministic tiebreaker)
          and proceed as follows.
          \begin{enumerate}[\upshape a.]
            \item
                  If $f \in S_1$, compute the closed subgroup $T_0 \colonequals \Stab_T(X)$ of $T$ (see \cite{Milne2017}*{Corollary~1.81}),
                  apply the algorithm to $X/T_0 \subset T/T_0$,
                  and return the pullback of the result along $T \to T/T_0$.
            \item
                  If $f \in S_2$,
                  let $K'=\Q(\zeta_{N/2})$,
                  write $P = ((-1)^{e_1},\dots,(-1)^{e_n})$ with $e_i \in \{0, 1\}$,
                  and put $Q = (\zeta_N^{e_1}, \dots, \zeta_N^{e_n})$, so that $\sigma(Q)/Q=P$.
                  Then $\Spec \sigma$ preserves $t_Q(X) \subset T$,
                  and taking quotients (or invariant coordinate rings) yields
                  $X' \subset T' = \GG_{m,K'}^n$.
                  Apply the algorithm to $X'$ in $T'$ over $K'$,
                  and return the pullback of the result along $T \stackrel{t_Q}{\to} T \to T'$.
            \item
                  If $f \in S_3$, then check whether $X$ is reducible.
                  If so, return the union of the torsion closures of
                  the irreducible components of $X$;
                  otherwise, return the reduced subscheme of $X$.
          \end{enumerate}
    \item
          If $X \not\subseteq f^{-1}X$ for every $f \in S$,
          then apply the algorithm to $X \intersect f^{-1}X$ for each $f \in S$
          and return the union of the results.
  \end{enumerate}
\end{algorithm}

\begin{theorem} \label{T:torsion closure algorithm}
  Theorem~\ref{algorithm:torsion closure} returns the torsion closure of $X$.
\end{theorem}

\begin{proof}
  We first verify termination.
  It suffices to check that no branch of the recursion
  can proceed to infinite depth.
  In step~1a, $f=t_P$ for some nontrivial $P$, and $P \in T_0$,
  so $T_0 \ne \{1\}$;
  thus step~1a cannot occur twice without an instance of step~2 in between.
  In step~1b, we replace $K$ with a smaller number field;
  thus along a given branch, step~1b cannot occur more than $[K:\QQ]$ times
  without an instance of step~2 in between.
  Finally, along each branch, after steps 1a and~1b occur for the last time,
  steps 1c and~2 can occur only finitely many times since $T$ is noetherian.

  We next verify correctness.
  The reduction in step~2 is valid by
  Lemma~\ref{L:covering of torsion points}.
  The reductions in steps~1a and~1b are valid
  since torsion closures respect field extension
  and pullback by isogenies or translations by torsion points.
  The reduction in the reducible case of step~1c is valid
  since torsion closures can be computed on irreducible components.
  Finally, in the irreducible case of step~1c,
  the reduced subscheme of $X$ equals the torsion closure, by the following lemma.
\end{proof}

\begin{lemma} \label{L:torsion odd order}
  Let $K$ be a number field.
  Fix $\tau \in \Aut K$, a torsion point $P \in T(K)$, and an integer $m \ge 2$.
  Let $f \colon T \to T$ be the $\Q$-morphism
  $x \mapsto (\Spec \tau) \circ t_P \circ [m]$.
  If $X$ is integral and $f(X) \subseteq X$, then $X$ equals its torsion closure.
\end{lemma}
\begin{proof}
  By replacing $f$ by an iterate, we may assume that $\tau=1$.
  Let $Z_1,\ldots,Z_r$ be the irreducible components of $X_{\Kbar}$.
  Since $X$ is integral, the $Z_i$ are Galois conjugates,
  so they have the same dimension.
  Since $f$ is finite, it maps $Z_1$ to some $Z_j$,
  and then the conjugates of $Z_1$ are mapped to the conjugates of $Z_j$,
  so $f$ induces a permutation of $\{Z_1,\ldots,Z_r\}$.
  By replacing $f$ by an iterate, we may assume that $f(Z_i)=Z_i$ for each $i$.
  By \cite{Hindry1988}*{Lemme~10},
  $Z_i$ is a translate of a subtorus of $T_{\Kbar}$.
  Since $f(Z_i)=Z_i$, it must be a torsion coset.
  Thus $X$ equals its torsion closure.
\end{proof}

\begin{remark} \label{R:torsion closures in SageMath}
  We have implemented a variant of Algorithm~\ref{algorithm:torsion closure} in SageMath.
  To speed up the algorithm, we incorporated the following modifications:
  \begin{itemize}
    \item When we detect that the defining ideal of $X$ contains a univariate polynomial, we factor this polynomial (after enlarging $K$ suitably) so that we can reduce the dimension of $T$.
    \item When applying step~1a, we first check whether there exists a positive-dimensional subtorus $T_1$ of $T$ such that $X$ arises by pullback from $T/T_1$.
          If so, we use $T_1$ in place of $T_0$; otherwise, we use
          $\langle t_P \rangle$ in place of $T_0$. This does not affect termination.
    \item When applying step~2 to $f \in S_3$,
          at the next level of recursion we replace $S$ by $\{f\}$.
          This does not affect correctness because
          this branch of the recursion needs to account for
          only torsion points of order not divisible by 4.
    \item We sometimes cut down $X$ based on its reduction modulo 2,
          as in Section~\ref{sec:mod 2 relations}.
  \end{itemize}
  In experiments,
  we compute torsion closures easily when $\dim X \leq 1$,
  with difficulty when $\dim X = 2$, and not at all when $\dim X \geq 3$.
  Our use of Gr\"obner bases makes it difficult to analyze the running time,
  but a similar algorithm using resultants was analyzed in \cite{Aliev-Smyth2012};
  its complexity is superexponential in the number of variables.
\end{remark}

\section{Low order solutions to the Gram determinant equation}
\label{sec:C code}

In this section, we prove the following statement.
\begin{proposition} \label{P:low order solutions}
  For $N \in \{48, 90, 120, 132, 168, 280, 420\}$, every $4$-line configuration with angles in $\ZZ \pi/N$ appears in some configuration accounted for in Theorem~\ref{T:main}.
\end{proposition}

To prove this, we make a rigorous computation of the solutions
$\Theta = (\theta_{12},\theta_{34},\theta_{13},\theta_{24},\theta_{14},\theta_{23}) \in \M_4(\R)^{\sym}_0$
to \eqref{eq:gram det 4x4}
with $\theta_{ij} \in \Z\pi/N$.
This computation combines numerical and algebraic methods.
In Section~\ref{sec:solutions to Gram equation},
a separate computation will show that these account for all solutions
outside of some specific families.

Write $\theta_{ij} = m_{ij} \pi/N$ with $m_{ij} \in \{1,\ldots,N-1\}$.
By exploiting $S_4$-symmetry, we may assume
\begin{align}  \label{eq:m_ij inequalities}
  \begin{split}
    m_{14} + m_{23} \; \leq \; m_{13} &+ m_{24} \; \leq \; m_{12} + m_{34}, \\
    m_{34} \; &\leq \; m_{12}, \\
    m_{24} \; &\leq \; m_{13};
  \end{split}
\end{align}
but we cannot also assume $m_{23} \leq m_{14}$.
The plan is to loop over $m_{12}, m_{34}, m_{13}, m_{24}, m_{14}$;
then \eqref{eq:gram det 4x4} expresses $\cos \theta_{23}$ as a root of
a quadratic equation,
so we can numerically solve for the possibilities for $m_{23}$,
rounding them to the nearest integer, and carry out three tests on the resulting $6$-tuple:
\begin{enumerate}[\phantom{mm}\upshape i)]
  \item We test whether~\eqref{eq:gram det 4x4} holds to within $10^{-11}$
        using \cplusplus\ and double precision arithmetic. (In a few cases, this requires computing $m_{23}$ using more working precision; see below.)

  \item If the first test passes,
        we test whether~\eqref{eq:gram det 4x4} holds to within $10^{-50}$
        using Bailey's \cplusplus\
        quad-double package~\cite{cite:bailey} (approximately 65 decimal digits of working precision).

  \item If the second test passes,
        we rigorously verify \eqref{E:105} and hence~\eqref{eq:gram det 4x4}
        by an algebraic computation in $\QQ(\zeta_{2N})$ in SageMath.
\end{enumerate}
To save time, we precompute the values of
$\cos(m\pi/N)$ and $\cos^2(m\pi/N)$, in both double and quad-double precision, for all $m<N$.

The first and second test were run on 
a MacBook Pro with a 2.9GHz Intel Core i7 CPU.
The case of $N=420$ dominated the time needed
and took one day of computation using one core of the CPU.
Most tuples were ruled out by the first test, so the second and third tests
took a negligible amount of time.

While the third test, being algebraic, confirms rigorously that we have no false positives, we must do some analysis to rule out false negatives. We state this in the form of a lemma.

\begin{lemma} \label{L:solutions up to 420}
  For $N \leq 420$, any tuple satisfying
  \eqref{eq:gram det 4x4} and~\eqref{eq:m_ij inequalities}
  passes tests 1 and~2.
\end{lemma}

\begin{proof}
  We begin with some observations about the accuracy of underlying floating-point arithmetic in our computations.
  Note that \cplusplus\ doubles correspond to IEEE-754 doubles, with 52 of 64 bits devoted to the mantissa,
  i.e., at least 15 decimal digits. Moreover, we compared the cosine values in doubles to quad doubles, finding agreement to within $10^{-15}$; this ensures the accuracy of the cosines in doubles to more than 50 bits.

  For a 6-tuple eliminated in the first test,
  we have to rule out a relative error of greater than $10^{-4}$
  (since the cosines are accurate to 15 decimal places, whereas only 11 decimal places are used to distinguish the
  determinant~\eqref{eq:gram det 4x4} from zero).
  The computation of the determinant from the matrix entries
  involves a few dozen multiplications and
  additions of cosines that are bounded in size by 1, and thus is quite safe
  \emph{provided} that $m_{23}$ is correctly computed from the other values.
  That is, let $\cos \theta_{23} $ be a root
  of $A x^2 +Bx + C$ with $A=\cos^2 \theta_{14} - 1$ and $B$ and $C$ given by more complicated expressions in the five other cosines;
  we must ensure that when we solve the quadratic equation
  to obtain
  \begin{equation} \label{eq:cosine as root of quadratic}
    m_{23} = \frac{N}{\pi} \arccos \frac{-B \pm (B^2 -4AC)^{1/2}}{2A},
  \end{equation}
  we are guaranteed to obtain $m_{23}$ to within $0.5$.
  In general this will not yield an integer, but we nonetheless round the computed value to the nearest integer and then test whether \eqref{eq:gram det 4x4} holds. (In most cases this is redundant because the true value of $m_{23}$ is not an integer, but this extra step takes negligible time due to our use of precomputed values, as described above.)

  We analyze the numerical stability of \eqref{eq:cosine as root of quadratic} by stepping through the computation.
  The denominator, $2A = 2(\cos^2 \theta_{14}-1)$, can be as small as $2(\cos^2(\pi/420)-1) = 0.00011\cdots$.
  Let $\alpha = -B \pm (B^2-4AC)^{1/2}$ and $\beta=2A$,
  and let $\alpha+\Delta_1, \beta+\Delta_2$ be the numerical
  values computed for $\alpha, \beta$; then
  \[
    \frac{\alpha}{\beta} - \frac{\alpha+\Delta_1}{\beta+\Delta_2} = \frac{\alpha \Delta_2 - \beta\Delta_1}{\beta(\beta+\Delta_2)}.
  \]
  The factor
  $\alpha \Delta_2/\beta^2$ can act to magnify the error;
  for $\alpha \approx 10$ and $\beta = 0.00011\cdots$, this is roughly $10^9$. Additionally, taking
  the arccos introduces a further factor of
  $(N/\pi) (1-x^2)^{-1/2}$ to the error, coming from the mean value theorem
  applied to $(N/\pi)\arccos(x)$; in the worst case $N = 420$ and $x \approx \pi/420$, this yields a factor of $\approx 18000$.
  We conclude that in a few cases, we may lose more than 13 decimal digits of accuracy; however, if
  $(|\alpha|+|\beta|)/|\beta|^2 < 10^8$, the previous analysis guarantees that $m_{23}$ is safely computed correctly using double precision.
  In the remaining cases, we recomputed $m_{23}$ in quad-double precision to confirm its value; in practice we only had to resort to quad doubles for this step, in total, for less than $1/2000$ of the cases examined, and this had a negligible impact on the overall runtime.
\end{proof}

After determining whether a given $6$-tuple 
$(m_{12}, m_{34}, m_{13}, m_{24}, m_{14}, m_{23})$
is a solution to~\eqref{eq:gram det 4x4}, we further check
whether condition 1 in Proposition~\ref{P:realizability of angle matrices}
holds. By Remark~\ref{R:1 and 2 minors}, this is the same as checking that the four
$3 \times 3$ principal minors of the Gram matrix in~\eqref{eq:gram det 4x4} are nonnegative.
To do so, we numerically compute the four minors using quad-double precision (65 digits precision), and declare each one to be nonnegative if its computed value is
greater than $-10^{-50}$; the following lemma shows that this test is rigorous.

\begin{lemma} \label{L:lower bound on cofactors}
  For a solution 
  $(m_{12}, m_{34}, m_{13}, m_{24}, m_{14}, m_{23})$
  of~\eqref{eq:gram det 4x4}
  with $0 < m_{ij} < N \leq 420$,
  if some $3 \times 3$ principal minor of the Gram matrix is nonzero, then its absolute
  value is greater than $10^{-50}$.
\end{lemma}
\begin{proof}
  Without loss of generality, consider the top left $3 \times 3$ minor; it is
  $1+2\cos \theta_{12} \cos \theta_{13} \cos \theta_{23}
    -\cos^2 \theta_{12} -\cos^2 \theta_{13} - \cos^2 \theta_{23} $,
  which equals
  \begin{equation} \label{E:3 by 3 minor factorization}
    4
    \sin \frac{\theta_{12} + \theta_{13} + \theta_{23}}{2}
    \sin \frac{\theta_{12} + \theta_{13} - \theta_{23}}{2}
    \sin \frac{\theta_{12} - \theta_{13} + \theta_{23}}{2}
    \sin \frac{-\theta_{12} + \theta_{13} + \theta_{23}}{2}.
  \end{equation}
  One can verify this identity using trigonometric identities, or by writing $\cos(t)$ as $(z+1/z)/2$ with $z=\exp(\cplxi t)$ and factoring the corresponding Laurent polynomial, but it is suggested
  by noticing that the above expression vanishes when any of the inequalities in Remark~\ref{R:triangle inequalities} are equalities.
  For $N \leq 420$, each factor, if nonzero, has absolute value at least
  $\sin(\pi/(2N)) \geq 0.0074799\cdots$.
  This yields the desired bound by a wide margin (even $10^{-10}$ would suffice).
\end{proof}

\begin{remark}
Equation~\eqref{E:3 by 3 minor factorization} also gives 
the square of the volume of a parallelepiped formed by three unit vectors
in terms of the angles between them.
\end{remark}

By Proposition~\ref{P:realizability of angle matrices},
any solution whose Gram matrix has nonnegative $3 \times 3$ principal minors
is realized by 4 vectors in $\RR^3$.
We then sort solutions according to whether the underlying 4 vectors are:
\begin{enumerate}
  \item[(i)]
        all lying on one plane;
  \item[(ii)]
        3 vectors lying on a plane and one vector perpendicular to that plane;
  \item[(iii)]
        3 vectors lying on a plane with the origin in their convex hull,
        and the fourth vector neither on nor perpendicular to the plane;
  \item[(iv)]
        outward normals to the faces of a tetrahedron;
  \item[(v)]
        none of the above.
\end{enumerate}
To test for condition (i), we check whether all four $3 \times 3$ principal minors
are zero (implying that any three of the vectors are coplanar).
Using Lemma~\ref{L:lower bound on cofactors} again, we can test this rigorously by computing in quad-double precision and verifying that their absolute values are less than $10^{-50}$.

To test for conditions (ii) and (iii), we first verify that exactly one of the four $3 \times 3$ principal minors is zero (and the other three positive), again using Lemma~\ref{L:lower bound on cofactors}; if so, then we must have 3 vectors lying on a plane and the fourth not lying on the plane.
We then test for (ii) and (iii) respectively by checking whether the values of $m_{ij}$ corresponding to angles including (respectively, not including) the fourth vector are all equal to $N/2$ (respectively, sum up to $N$).

To test for condition (iv), we first verify that all four $3\times 3$ principal minors are positive, again using Lemma~\ref{L:lower bound on cofactors}, to ensure that the 4 vectors are in linear general position. In this case, condition (iv) asserts that the unique (up to scalar combination) vanishing linear combination of the 4 vectors has coefficients all of the same sign; using Cramer's rule, we check this by
computing the signs of the $3 \times 3$ nonprincipal minors. For this, we need
an analogue of Lemma~\ref{L:lower bound on cofactors} to reduce to a computation in quad doubles.

\begin{lemma} \label{L:lower bound on cofactors2}
  For a solution 
  $(m_{12}, m_{34}, m_{13}, m_{24}, m_{14}, m_{23})$
  of~\eqref{eq:gram det 4x4}
  with $0 < m_{ij} < N \leq 420$,
  if all diagonal cofactors of the Gram matrix are positive, 
  then every off-diagonal cofactor has absolute value greater than $10^{-50}$.
\end{lemma}

\begin{proof}
  Consider four unit vectors in $\RR^3$ given by Proposition~\ref{P:realizability of angle matrices}, and form a $3 \times 4$ matrix $B$ with these as column vectors. Let $B_j$ denote the $3 \times 3$ matrix obtained from $B$ by removing column $j$.
  Then the $3 \times 3$ submatrix of the Gram matrix obtained by removing row $i$ and
  column $j$ equals $B_i^T B_j$, so the corresponding cofactor
  equals $(-1)^{i+j}\det(B_i) \det(B_j)$. Up to sign, this is the geometric mean of two diagonal cofactors of the Gram matrix; we may thus deduce the claim directly from
  Lemma~\ref{L:lower bound on cofactors}.
\end{proof}

With this classification in hand,
we discard solutions of type (i) and (ii) as trivial cases.
We further filter solutions of type (iii) and (iv) for solutions
in a known parametric family. (We ignore solutions of type (v); any such solution arises from a solution of type (iii) or (iv) by negating one or more vectors.)
The remaining solutions are all accounted for by Theorem~\ref{T:main}, so
the proof of Proposition~\ref{P:low order solutions} is complete.

\begin{remark}
  Since it did not take much extra effort, we ran our code for all $N \leq 280$ and $N=420$.
  The extra values of $N$ provide a sanity check for the correctness of the implementation.
  As an additional sanity check, we
  use the unfiltered solutions of types (iii) and (iv) to experimentally find two-parameter solutions, and then one-parameter solutions not contained in a two-parameter solution. These agree with the solutions computed algebraically as described in Section~\ref{sec:solutions to Gram equation}.
  We do this by by looping over all triples of solutions to~\eqref{eq:gram det 4x4} found across several stretches of $N$ (such as $N<100$). Any such triple determines
  a plane in $\RR^6$. We then select 5 random points on each plane and test
  whether equation~\eqref{eq:gram det 4x4} holds for all five points, to within $10^{-11}$. If it does, we declare the three points to be part of a two-parameter family of solutions, and then confirm that the family matches one of the two-parameter families found algebraically (conversely, every two-parameter family found algebraically was confirmed in this fashion).
  Using exact rational arithmetic, we remove all solutions from our list that are on that plane. After exhausting all triples, we repeat the process with all remaining pairs of solutions to experimentally determine the one-parameter families of solutions, and verify that they
  match those found in Section~\ref{sec:solutions to Gram equation}. The remaining solutions are thus the sporadic solutions. The sporadic solutions of type (iv) are listed in Table~\ref{Table:tetrahedra}.

\end{remark}

\section{The 4-line configurations}
\label{sec:solutions to Gram equation}

In this section, we prove the following result in the direction of
Theorem~\ref{T:main}, then use this to deduce Theorem~\ref{T:tetrahedra}.

\begin{theorem} \label{T:4 lines classification}
  Every rational-angle $4$-line configuration, up to equivalence, is contained in
  one of the configurations indicated in Theorem~\ref{T:main}.
\end{theorem}

Our approach is to combine the computational results of Section~\ref{sec:C code} 
with a partial classification of solutions of \eqref{eq:gram det 4x4}, 
initially done modulo the symmetries identified in 
Sections~\ref{section:algebraic tori} and~\ref{section:Regge symmetry}.
While in principle it is not necessary
to rely on the exhaustive computations, doing so makes the computations far more efficient and the results less vulnerable to programming errors.

\begin{definition}
  Let $\Lambda^*$ be the kernel of the homomorphism
  $\tau \circ \exp \colon \RR^6 \cong \M_4(\R)^{\sym}_0 \to T(\CC)$ from
  \eqref{E:spaces}; this is a lattice containing $(2\pi \Z)^6$ with index 8.
  Let $G$ be the group of affine-linear transformations of
  $\M_4(\R)^{\sym}_0$ generated by
  \begin{itemize}
    \item the translation action of $\Lambda^*$;
    \item the action of $\{\pm 1\}^6$ by multiplication on coordinates; and
    \item the action of $S_4$.
  \end{itemize}
  Let $G'$ be the group generated by $G$ and $\Regge$; note that $G'$ acts on $\calH$.
  A calculation shows that there is a short exact sequence
  $1 \to \Lambda^* \to G' \to W(D_6) \to 1$.
\end{definition}

\begin{lemma} \label{L:4 lines classification not in general position}
  Let $\Theta \in \calP \cap \calH \cap (\QQ \pi)^6$ be a matrix corresponding
  to a configuration of four lines, exactly three of which are coplanar. Then
  at least one of the following conditions holds:
  \begin{itemize}
    \item
          $\Theta$ is $G'$-equivalent to a matrix corresponding to a perpendicular $4$-line configuration (a configuration with one line perpendicular to the other three); 
    \item $\Theta$ is $G'$-equivalent to a matrix of one of the forms
          \begin{equation} \label{eq:one-parameter families}
            (x, x, 2\pi/3, \pi-2x, \pi/2, x), \,
            (x, 2x, 2\pi/3, \pi-3x, 2\pi/3, x), \,
            (\pi/3, 2x, \pi/2, \pi-3x, \pi+x, x)
          \end{equation}
          for some $x \in \QQ \pi$;
    \item $\Theta$ has entries in $\ZZ\pi /N$ for some $N \in \{90, 120, 132, 168, 280, 420\}$.
  \end{itemize}
\end{lemma}

\begin{proof}
  We first compute $G$-orbit representatives 
  for the set of matrices $\Theta$ as above.
  Each $G$-orbit contains a matrix of the form
  $\Angle (\vv_1,\vv_2,\vv_3,\vv_4)$
  where $\vv_2,\vv_3,\vv_4$ are coplanar with $0$ in their convex hull, so that
  \begin{equation} \label{eq:orientation condition}
    \theta_{23} + \theta_{24} + \theta_{34} \equiv 0 \pmod{2 \pi};
  \end{equation}
  it is thus equivalent to compute $G_1$-orbits of such matrices,
  where $G_1$ is the subgroup of $G$ that preserves \eqref{eq:orientation condition}.
  Substituting $z_{23} = z_{24}^{-1} z_{34}^{-1}$ into \eqref{E:105} 
  yields a square; taking a square root leads to
  \begin{equation} \label{eq:square root of gram equation}
    \begin{aligned}
      \cos (\tfrac{\pi}{2} +\theta_{12}-\theta_{34})
      + \cos (\tfrac{\pi}{2} +\theta_{13}-\theta_{24})
      + \cos (\tfrac{\pi}{2} + \theta_{14}-\theta_{23}) & \\
      + \cos (\tfrac{\pi}{2} - \theta_{12}-\theta_{34})
      + \cos (\tfrac{\pi}{2} - \theta_{13}-\theta_{24})
      + \cos (\tfrac{\pi}{2} - \theta_{14}-\theta_{23}) &= 0.
    \end{aligned}
  \end{equation}
  Theorem~\ref{T:sums of 6 cosines} implies that any solution of \eqref{eq:square root of gram equation} is a specialization of a combination of indecomposable relations of
  one of the following forms in the notation of Table~\ref{table:sums of cosines}, up to the transformations in Theorem~\ref{T:sums of 6 cosines}:
  \begin{gather} \label{eq:forms of 6-term relations2}
    6, 5+1, \\
    \label{eq:forms of 6-term relations1}
    2^*+ 2^*+ 2^*, 3^*+ 3^*, 3^*+ 2^*+ 1, 6^*, 5^*+ 1, 4+2^*, 3^*+ 3, 3+ 2^*+ 1;
  \end{gather}
  here \eqref{eq:forms of 6-term relations2} lists the possibilities with no free parameters.
  We omit forms including $1+1$
  because such a pair is a specialization of a relation of type $2^*$.
  Similarly, we omit $3+3$.

  Table~\ref{table:sums of cosines} shows that
  any solution of \eqref{eq:square root of gram equation}
  of a form in \eqref{eq:forms of 6-term relations2}
  has values in $\ZZ \pi/N$ for some
  $N \in \{132, 168, 280, 420\}$.
  Thus it remains to identify solutions of the equations \eqref{eq:orientation condition}
  and \eqref{eq:square root of gram equation} corresponding to forms listed in \eqref{eq:forms of 6-term relations1}, modulo the action of $G_1$.
  The solutions will fall into finitely many families,
  each represented as the set of solutions to a system of congruences
  $A \Theta = 2\pi \bb \pmod{2\pi}$ for some integer matrix $A$
  and rational vector $\bb$.
  To put a collection of augmented matrices $(A|\bb)$ into a standard form,
  we perform the following operations until they have no further effect.
  \begin{itemize}
    \item Perform row reduction on each $(A|\bb)$ to put $A$ into Hermite normal form, omitting zero rows.
    \item If a row $(a_1 \cdots a_6 | b)$ of some $(A|\bb)$ has $d \colonequals \gcd(a_1,\dots,a_6) > 1$, replace $(A|\bb)$ with the $d$ matrices obtained by replacing this row in turn by $(\tfrac{a_1}{d} \cdots \tfrac{a_6}{d} | \tfrac{b+i}{d})$ for $i=0,\dots,d-1$.
    \item Reduce the coordinates of each $\bb$ modulo $1$ to put them in $[0,1)$.
  \end{itemize}
  To intersect two families $(A_1|\bb_1)$ and $(A_2|\bb_2)$,
  perform the first operation on
  $\left(
    \begin{array}{@{}c|c@{}}
        A_1 & \bb_1 \\
        A_2 & \bb_2
      \end{array}
    \right)$.
  To test whether one family is contained in another,
  compare it to their intersection.

  For each form in \eqref{eq:forms of 6-term relations1}
  and cosine relation of that form,
  each possible matching of this relation to the six angles
  \[
    \tfrac{\pi}{2} \pm \theta_{12}-\theta_{34}, \qquad
    \tfrac{\pi}{2} \pm \theta_{14}-\theta_{23}, \qquad
    \tfrac{\pi}{2} \pm \theta_{14}-\theta_{23},
  \]
  defines an augmented matrix $(A|\bb)$ as above.
  We put these matrices into standard form,
  eliminate any family contained in another family,
  and eliminate any family contained in one of the degenerate families
  \begin{gather*}
    \theta_{jk} \equiv 0, \pi \pmod{2 \pi} \\
    \theta_{1j} \pm \theta_{1k} \pm \theta_{jk} \equiv 0 \pmod{2 \pi};
  \end{gather*}
  the latter corresponds to $\vv_1,\vv_j,\vv_k$ being coplanar
  in addition to $\vv_2,\vv_3,\vv_4$, making all four coplanar.
  These computations take about 2 hours in SageMath
  on a virtual 2.3GHz Intel Xeon CPU.

  For each matrix $(A|\bb)$ in the result of the computation,
  we may solve the equation $A\vv = \bb$
  to obtain an affine subspace of $\calH$.
  Inspecting the output, we find the following.
  \begin{itemize}
    \item We obtain no subspaces of dimension three or more.
    \item We obtain $3$ two-dimensional subspaces,
          which belong to a single $G'$-orbit. This orbit contains a subspace consisting entirely of perpendicular $4$-line configurations.
    \item We obtain $13$ one-dimensional subspaces, which belong to 3 distinct $G'$-orbits. These orbits are represented by the three subspaces listed in \eqref{eq:one-parameter families}.
    \item The remaining subspaces are isolated points with coordinates in $\ZZ \pi/N$ for some $N \in \{84, 90, 120\}$.
          Any $\Theta$ that is $G$-equivalent to one of these also has coordinates in $\ZZ \pi/N$.\qedhere
  \end{itemize}
\end{proof}

\begin{remark}
  The equations \eqref{eq:orientation condition} and \eqref{eq:square root of gram equation}
  constitute the same system as the one solved in \cite{Poonen-Rubinstein1998}*{Theorem~4.4} to classify concurrent diagonals of regular polygons,
  except for some positivity conditions in the latter statement. Unfortunately, these conditions prevent us from deriving Lemma~\ref{L:4 lines classification not in general position} directly from results in \cite{Poonen-Rubinstein1998}.
\end{remark}

\begin{lemma} \label{L:4 lines classification general position}
  Let $\Theta \in \calP \cap \calH \cap (\QQ \pi)^6$ be a matrix corresponding
  to a configuration of four lines, no three of which are coplanar. Then
  at least one of the following conditions holds:
  \begin{itemize}
    \item $\Theta$ is $G'$-equivalent to a matrix of the form
          \begin{equation} \label{eq:one-parameter families2}
            (\pi/2, \pi/2, \pi-2x, \pi/3, x, x)
          \end{equation}
          for some $x \in \QQ \pi$;
    \item $\Theta$ has entries in $\ZZ\pi /N$ for some $N \in \{48, 120, 132, 168, 280, 420\}$.
  \end{itemize}
\end{lemma}
\begin{proof}
  Multiplying the matrix in \eqref{eq:gram det 4x4} by $2$
  and reducing modulo $2\ZZ[\mu]$ yields the congruence
  \begin{equation} \label{eq:mod 2 cosine condition}
    \begin{aligned}
      2\cos (2\theta_{12} + 2\theta_{34}) +
      2\cos (2\theta_{13} + 2\theta_{24}) +
      2\cos (2\theta_{14} + 2\theta_{24}) &
      \\
      {} + 2\cos (2\theta_{12} - 2\theta_{34}) +
      2\cos (2\theta_{13} - 2\theta_{24}) +
      2\cos (2\theta_{14} - 2\theta_{24}) & \equiv 0 \pmod{2\ZZ[\mu]}.
    \end{aligned}
  \end{equation}
  Theorem~\ref{T:sums of 12 mod 2}
  classifies solutions to \eqref{eq:mod 2 cosine condition}
  in terms of the indecomposable mod~2 relations listed
  in Table~\ref{table:sums of cosines},
  which include a length~$1$ relation which we call $1a$,
  and the additional length~$1$ relation $2 \cos 0 \equiv 0$
  in Theorem~\ref{T:sums of 12 mod 2}, which we call $1b$;
  let $1$ denote a relation of type $1a$ or $1b$.
  Explicitly, any solution is a specialization of a
  sum of relations of the following forms,
  up to the transformations in Theorem~\ref{T:sums of 12 mod 2}:
  \begin{gather}
    \label{eq:forms of 6-term relations4}
    6, \; 5+1, \; 4+1a+ 1b; \\
    \label{eq:forms of 6-term relations3}
    2^*+2^*+2^*, \; 3^*+3^*, \; 3^*+2^*+1, \; 2^*+2^*+1a+1b, \; 6^*, \; 5^*+1, \; 4+2^*, \; 3^*+ 3, \; 3+2^*+1;
  \end{gather}
  here \eqref{eq:forms of 6-term relations4} lists the possibilities with no free parameters.
  We omit forms including $1a+1a$ or $1b+1b$ 
  because such pairs are specializations of a relation of type $2^*$. 
  Similarly, we omit $3+3$.

  Table~\ref{table:sums of cosines} shows that any solution of \eqref{eq:mod 2 cosine condition} of a type in \eqref{eq:forms of 6-term relations4}
  has values in $\ZZ \pi/N$ for some $N \in \{120, 132, 168, 280, 420\}$.
  It thus remains to identify solutions of \eqref{eq:gram det 4x4} arising from
  solutions of \eqref{eq:mod 2 cosine condition} of forms listed in \eqref{eq:forms of 6-term relations3}. For each such form, each mod 2 cosine relation of that form, and each possible matching of this relation to the six angles
  \[
    2 \theta_{12} \pm 2 \theta_{34}, \qquad
    2 \theta_{13} \pm 2 \theta_{24}, \qquad
    2 \theta_{14} \pm 2 \theta_{23},
  \]
  we obtain an ideal in the Laurent polynomial ring
  $\QQ[z_{ij}^{\pm}]$ via the substitution $z_{jk} = e^{\cplxi \theta_{jk}}$.
  At this point we have the option to replace this set of ideals with a set of $G'$-orbit representatives; this turns out to be worthwhile
  for forms containing two or more free parameters.
  We then impose the condition \eqref{eq:gram det 4x4}
  by adding the generator \eqref{E:105} to each ideal.

  Using the implementation in Remark~\ref{R:torsion closures in SageMath},
  we compute the torsion closures of the corresponding varieties.
  Once this is done, we eliminate solutions which are degenerate because one of the angles equals $0$ or $\pi$ (corresponding to two of the lines coinciding)
  or because one of the $3 \times 3$ minors of the Gram matrix vanishes
  (corresponding to the four lines not being in linear general position).
  These computations take about $2$ hours in SageMath 
  on a virtual 2.3GHz Intel Xeon CPU.

  Each irreducible component of each torsion closure in the output
  corresponds to a subset of $\calH$
  consisting of the $(2 \pi \ZZ)^6$-translates of some affine subspace of $\calH$.
  Inspecting the output, we find the following.
  \begin{itemize}
    \item We obtain no subspaces of dimension two or more.
    \item The one-dimensional subspaces belong to the $G'$-orbit of the subspace listed in \eqref{eq:one-parameter families2}.
    \item The remaining subspaces are isolated points
          with coordinates in $\ZZ \pi/N$ for $N \in \{21, 24, 60\}$.
          Any element of the $G'$-orbit of one of these points
          has coordinates in $\ZZ \pi/(2N)$.\qedhere
  \end{itemize}
\end{proof}

For each $\Theta$ in 
Lemmas~\ref{L:4 lines classification not in general position}
and~\ref{L:4 lines classification general position},
we now determine which elements in its $G'$-orbit lie in $\calP$.

\begin{lemma} \label{L:expand orbits}
  The $G'$-orbit of each subspace in \eqref{eq:one-parameter families} 
  or \eqref{eq:one-parameter families2} 
  yields a single $\Regge^{\pm}$-orbit of one-parameter families
  of $\R^3$-realizable $4 \times 4$ rational-angle matrices.
  Each orbit is represented by a family yielding a $4$-line subconfiguration 
  of a $6$-line configuration
  in Example \ref{exa:maximal parametric 6-line configurations2}
  or~\ref{exa:maximal parametric 6-line configurations1}, respectively.
\end{lemma}

\begin{proof}
We have $G' = \Regge^{\pm} \cdot (2 \pi \ZZ)^6 \cdot \{\pm 1\}^6$; 
for example,
the translation in $\Lambda^*$ 
sending $\theta_{1j}$ to $\pi + \theta_{1j}$ for $j=2,3,4$
is the composition of an element of $\{\pm 1\}^6$ with the element of $\Regge^{\pm}$ sending $\theta_{1j}$ to $\pi - \theta_{1j}$ for $j=2,3,4$.
Also, $\Regge^{\pm}$ preserves $\calP$.
Therefore it suffices to determine, for each subspace $V$
in the $\{\pm 1\}^6$-orbit of a subspace in 
\eqref{eq:one-parameter families} or \eqref{eq:one-parameter families2},
which $\mathbf{a} \in (2\pi\ZZ)^6$ are such that 
$\mathbf{a}+V$ intersects $\calP$.
Each $V$ has a parametrization $\theta(t) = t \vv + \ww$ with $\vv \in \Z^6$;
then $\theta(t+2\pi) \equiv \theta(t) \pmod{(2\pi\ZZ)^6}$,
so it suffices to consider the finitely many $\mathbf{a}$
such that $\mathbf{a} + \theta([-\pi,\pi])$ intersects $[0,\pi]^6 \supset \calP$.
We keep each such $\mathbf{a}+V$ for which 
$(\mathbf{a}+V) \intersect \calP \ne \emptyset$.
Finally, we compute the $\Regge^{\pm}$-orbit representatives.
This computation takes about $25$ minutes in SageMath 
on a 2.3GHz Intel Core i5 CPU.
\end{proof}

\begin{proof}[Proof of Theorem~\ref{T:4 lines classification}]
  Let $\calF$ be a two-parameter family of matrices associated to 
  the perpendicular $4$-line configurations.
  By Lemmas \ref{L:4 lines classification not in general position} 
  and~\ref{L:4 lines classification general position},
  any family of realizable $4 \times 4$ rational-angle matrices 
  with two or more parameters is in the $\Regge^{\pm}$-orbit of $\calF$.
  Operating by coset representatives of $S_4^{\pm} \subseteq \Regge^{\pm}$ 
  shows that this $\Regge^{\pm}$-orbit is a union of three $S_4^{\pm}$-orbits;
  one is that of the (nonmaximal) family $\calF$,
  and the other two are the maximal two-parameter families
  in Theorem~\ref{T:main}.
  By Lemmas~\ref{L:4 lines classification not in general position} 
  and \ref{L:4 lines classification general position}
  combined with Lemma~\ref{L:expand orbits}, 
  every one-parameter family is in the $\Regge^{\pm}$-orbit of 
  one of four families; again we split each $\Regge^{\pm}$-orbit
  into $S_4^{\pm}$-orbits 
  and find that they are as described in Theorem~\ref{T:main}.
  Finally, by Lemmas~\ref{L:4 lines classification not in general position}
  and \ref{L:4 lines classification general position} 
  combined with Proposition~\ref{P:low order solutions},
  the isolated configurations are as described by Theorem~\ref{T:main}.
\end{proof}

\begin{proof}[Proof of Theorem~\ref{T:tetrahedra}]
For this, we do not need the configurations described in 
Lemma~\ref{L:4 lines classification not in general position}.
We need only the $\Regge^{\pm}$-orbit in Lemma~\ref{L:expand orbits}
from \eqref{eq:one-parameter families2}
and the isolated configurations of Proposition~\ref{P:low order solutions}.
We compute representatives for the $S_4$-orbits and filter by the 
``test for condition~(iv)'' in the paragraph 
before Lemma~\ref{L:lower bound on cofactors2}
(for the one-parameter families, 
it turns out that the parameter range corresponding to 
configurations with no three lines coplanar is an open interval, 
so it suffices to check one interior sample point, 
by continuity of the signs of the minors).
The result is that 
the $\Regge^{\pm}$-orbit from \eqref{eq:one-parameter families2}
yields the two infinite families in Theorem~\ref{T:tetrahedra},
and the isolated configurations outside those 
yield the list in Table~\ref{Table:tetrahedra}.
\end{proof}

\section{From 4-line configurations to \texorpdfstring{$n$}{n}-line configurations}
\label{sec:more lines}

We now complete the proof of Theorem~\ref{T:main},
by assembling rational-angle $n$-line configurations
from the classification of 4-line configurations given by Theorem~\ref{T:4 lines classification}.
This requires some care to make the computation feasible.

\begin{proposition}
  \label{P:contained in perpendicular}
  Let $n \ge 4$.
  Let $\calL$ be an $n$-line configuration in $\R^3$.
  Then $\calL$ is contained in a perpendicular configuration
  if and only if each $4$-line subconfiguration of $\calL$ is contained in
  a perpendicular configuration.
\end{proposition}

\begin{proof}
  Suppose that each $4$-line subconfiguration is contained in a perpendicular
  configuration.
  Then for every four lines in $\calL$, there is a unique plane containing
  at least three of them, and the fourth is either in the plane or perpendicular
  to it.
  Fix $L_1,L_2,L_3 \in \calL$ lying in a plane $P$.
  Then the unique plane for $\{L_1,L_2,L_3,L_i\}$ for any other $L_i \in \calL$
  must be $P$, and $L_i$ is either in $P$ or is the line perpendicular to $P$.
  Thus $\calL$ is contained in a perpendicular configuration.
\end{proof}

For each $n \ge 4$,
let $\scriptM_n$ be the set of $\R^3$-realizable $n \times n$ rational-angle matrices.

\begin{proposition}
  \label{P:finiteness}
  For each $n \ge 4$,
  there exists a finite set $\scriptA_n$ of
  affine $\Q$-subspaces of $\M_n(\Q \pi)^{\sym}_0$
  such that $\scriptM_n = \Union_{A \in \scriptA_n} (A \intersect \calP_n)$.
\end{proposition}

\begin{proof}
  The calculations described in the preceding sections construct such a set $\scriptA_n$
  when $n=4$.

  Now suppose that $n>4$.
  Let $\binom{[n]}{4}$ be the set of $4$-element subsets of $\{1,\ldots,n\}$.
  For $I \in \binom{[n]}{4}$, let $p_I \colon \M_n(\Q)^{\sym}_0 \to \M_4(\Q)^{\sym}_0$
  be the projection giving the principal submatrix indexed by $I$.
  By Corollary~\ref{C:realizable}, for each $n>4$,
  a matrix $\Theta \in \M_n(\R)$ is in $\scriptM_n$
  if and only if its $4 \times 4$ principal submatrices are in $\scriptM_4$.
  Thus we may take $\scriptA_n$ to be the set of nonempty intersections
  of the form $\Intersection_{I} p_I^{-1}(A_I)$
  where each $A_I$ ranges over $\scriptA_4$ independently.
\end{proof}

We may assume that each $A \in \scriptA_n$
equals the affine span of $A \intersect \calP_n$.
We may also assume that $\scriptA_n$ is irredundant
in the sense that if $A,A' \in \scriptA_n$ satisfy $A \subset A'$,
then $A=A'$.
These conditions specify $\scriptA_n$ uniquely.

Let $\scriptA_n'$ be the set of $A \in \scriptA_n$ such that
$A \intersect \calP_n$ contains a matrix with no off-diagonal entries
equal to $0$ or $\pi$.
Let $\scriptA_n''$ be the set of $A \in \scriptA_n'$ such that
$A \intersect \calP_n$ contains a matrix corresponding to a line configuration
not contained in a perpendicular configuration
(for each $n$,
this condition removes exactly $n+1$ elements:
the one parametrizing planar configurations,
and, for each $i$,
the one parametrizing configurations with the $i$th vector perpendicular to all
the others).
The group $S_n^\pm$ acts on $\scriptA_n$, $\scriptA_n'$, and $\scriptA_n''$.
To prove Theorem~\ref{T:main},
we need to compute the set $\scriptR_n''$
of $S_n^\pm$-orbit representatives in $\scriptA_n''$
for $n$ up to $16$
and verify that $\scriptR_{16}'' = \emptyset$
(we need $16$, and not just $15$,
to rule out adding a $16$th line to the icosidodecahedral configuration).
By Proposition~\ref{P:contained in perpendicular},
we need only consider angle matrices whose upper left $4 \times 4$ submatrix
is in $\scriptR_4''$.

What complicates our task is that $\#\scriptA_4' = 84696$.
It is not practical to loop over all $84696^{\binom{n}{4}}$ tuples $(A_I)$
as suggested by the proof of Proposition~\ref{P:finiteness},
even for $n=5$, let alone $n=16$.

One could imagine computing a set $\scriptR_n''$ by induction,
using projections onto $(n-1) \times (n-1)$ principal submatrices instead of
$4 \times 4$ principal submatrices.
Suppose that $\scriptR_{n-1}''$ is known.
The action of $S_{n-1}^\pm \subset S_n^\pm$
shows that each $S_n^\pm$-orbit in $\scriptA_n''$
has a representative with upper left $(n-1) \times (n-1)$ submatrix in $\scriptR_{n-1}''$,
but we do not have the freedom to assume simultaneously that the other
$(n-1) \times (n-1)$ principal submatrices are in $\scriptR_{n-1}''$;
all we can assume is that they are in $\scriptA_{n-1}'$.
This is a problem, since $\scriptA_{n-1}'$ for some of the larger values of $n$
is much larger even than $\scriptA_4'$ because even a single $S_{n-1}^\pm$-orbit
can be huge.

Therefore instead we employ the following ``early abort'' inductive strategy.
Start with the list $\scriptR_{n-1}''$ of affine subspaces giving possibilities for
the upper left $(n-1) \times (n-1)$ submatrix.
In the first stage, fix $I \in \binom{[n]}{4} - \binom{[n-1]}{4}$
and try to reconcile
each possibility for the upper left $(n-1) \times (n-1)$ submatrix
with each possibility in $\scriptA_4'$ for the $I \times I$ principal submatrix ---
this amounts to intersecting preimages of affine subspaces, as in the proof
of Proposition~\ref{P:finiteness}.
Most of these preimage intersections will be empty
or will correspond to a family whose general member
has an off-diagonal entry equal to $0$ or $\pi$,
so they need not be considered further;
later on in the process we will also have intersections that are reduced to a point,
and we can discard those too if the point happens not to satisfy the inequalities
defining $\calP_n$.
In the second stage, choose a different $I' \in \binom{[n]}{4} - \binom{[n-1]}{4}$
and try to reconcile the undiscarded possibilities with the
possibilities for the $I' \times I'$ principal submatrix
by computing preimage intersections again.
There are $84696$ branches at each stage, but most of the branches abort immediately,
and it turns out that the list of possibilities remains under control.
After completing a stage for every subset in $\binom{[n]}{4} - \binom{[n-1]}{4}$,
we have a list of affine subspaces whose $S_n^{\pm}$-orbits include all
the subspaces in $\scriptA_n''$.
We then compute a distinguished representative of the $S_n^\pm$-orbit of
each subspace and eliminate redundancies, to obtain $\scriptR_n''$.

\begin{remark}
  To save more time, one can totally order the
  set of $S_4^\pm$-orbits in $\scriptA_4'$,
  with the ones represented by elements of $\scriptR_4''$ coming first.
  This induces a pre-order on $\scriptA'_4$ itself.
  Then, by acting by $S_n^\pm$, we may assume that for each subspace in $\scriptR_n''$,
  obtained as the intersection of preimages of $A_I$,
  the subspace $A_I$ for $I=\{1,2,3,4\}$ is less than or equal to the
  $A_J$ for every other $J \in \binom{[n]}{4}$.
  Thus when seeding the inductive process with a particular $A_{\{1,2,3,4\}}$,
  we need only consider $A_J$ that are greater than equal to that one
  in each stage.
  By choosing the total ordering judiciously, starting with affine subspaces
  corresponding to line configurations that are unlikely to extend much,
  we greatly reduce the number of branches in stages for larger $n$.
  In fact, for simplicity we use a total pre-order instead of a total order;
  in other words, we group the orbits into clumps, and totally order the clumps.
  These improvements reduce the running time of all
  the calculations in this section to a total of
  14 hours in Magma on a 3.5GHz Intel Xeon CPU E5-1620 v3.
\end{remark}

\begin{example}
  Each of the five $8$-line configurations
  consists of seven of the central diagonals of a $60$-gon centered at $\mathbf{0}$
  together with one line neither in its plane nor perpendicular to it.
  Each of the five configurations has a different angle set,
  though in each case the angles are among the multiples of $\pi/30$.
\end{example}

\begin{example} \label{exa:maximal parametric 6-line configurations1}
  One of the one-parameter families of $6$-line configurations is
  obtained by taking the lines spanned by $(1,0,0)$ and
  $\Bigl( 0,-\frac{2}{\sqrt{3}} \cos \theta, \sqrt{1 - \tfrac{4}{3} \cos^2 \theta} \Bigr)$
  and their rotations by $\pm 2\pi/3$ about the $z$-axis,
  for each parameter value
  $\theta \in \QQ \pi \intersect (\pi/6, \pi/2)$.
  The angles formed are
  $\pi/2$, $2\pi/3$, $\theta$, $\pi-\theta$, and $\pi-2\theta$.
\end{example}

\begin{example} \label{exa:maximal parametric 6-line configurations2}
  Three more one-parameter families of $6$-line configurations
  can be obtained by taking the lines spanned by
  $(\cos r \theta, \sin r \theta, 0)$ for $r \in \{-2,-1,0,1,2\}$
  and $\Bigl(0, \tfrac{1}{2} \csc \theta, \sqrt{1 - \tfrac{1}{4} \csc^2 \theta } \Bigr)$,
  for $\theta \in \QQ \pi$ in one of the parameter ranges $(\pi/6, \pi/4)$, $(\pi/4,\pi/3)$, or $(\pi/3,\pi/2)$.
  The angles formed are $\pi/3$, $\pi/2$, $2\pi/3$, $\theta$, $2\theta$, $3\theta$,
  and $4\theta$ (apply $x \mapsto 2\pi-x$ if they exceed $\pi$).
\end{example}

\section{Tables}
\label{sec:tables}

We tabulate our results in a somewhat compressed form.
A more verbose description can be found in the GitHub repository 
mentioned near the end of Section~\ref{S:introduction}.

\subsection{Sporadic tetrahedra}

Table~\ref{Table:tetrahedra} lists the 59 similarity classes of tetrahedra with rational dihedral angles not belonging to one of the two parametric families described in
Theorem~\ref{T:tetrahedra}. Each entry in the table lists the dihedral angles $(\alpha_{12}, \alpha_{34}, \alpha_{13}, \alpha_{24}, \alpha_{14}, \alpha_{23})$ measured in units of $\pi/N$ for the integer $N$ listed in the left column. The horizontal lines indicate groupings into orbits for the group $\Regge$ generated by Regge symmetries (see Section~\ref{section:Regge symmetry}).
For those tetrahedra listed in \cite{Boltianskii1978}*{pp.~170--173}, we have included the labels used therein; all of these correspond to rational-angle 4-line configurations contained in either the 9-line or the 15-line maximal configuration.

\begin{table}[ht]

  \small
  \begin{tabular}{c|c}
    $N$ & $(\alpha_{12}, \alpha_{34}, \alpha_{13}, \alpha_{24}, \alpha_{14}, \alpha_{23})$  as multiples of $\pi/N$
    \\
    \hline
    \hline
    12  & $(3, 4, 3, 4, 6, 8) = H_2(\pi/4)$                                                                         \\
    24  & $(5, 9, 6, 8, 13, 15)$                                                                                    \\
    \hline
    12  & $(3, 6, 4, 6, 4, 6) = T_0$                                                                                \\
    24  & $(7, 11, 7, 13, 8, 12)$                                                                                   \\
    \hline
    15  & $(3, 3, 3, 5, 10, 10) = T_{18}$, $(2, 4, 4, 4, 10, 10)$, $(3, 3, 4, 4, 9, 11)$                            \\
    \hline
    15  & $(3, 3, 5, 5, 9, 9) = T_7$                                                                                \\
    \hline
    15  & $(5, 5, 5, 9, 6, 6) = T_{23}$, $(3, 7, 6, 6, 7, 7)$, $(4, 8, 5, 5, 7, 7)$                                 \\
    \hline
    21  & $(3, 9, 7, 7, 12, 12)$, $(4, 10, 6, 6, 12, 12)$, $(6, 6, 7, 7, 9, 15)$                                    \\
    \hline
    30  & $(6, 12, 10, 15, 10, 20) = T_{17}$, $(4, 14, 10, 15, 12, 18)$                                             \\
    60  & $(8, 28, 19, 31, 25, 35)$, $(12, 24, 15, 35, 25, 35)$,                                                    \\& $(13, 23, 15, 35, 24, 36)$, $(13, 23, 19, 31, 20, 40)$
    \\
    \hline
    30  & $(6, 18, 10, 10, 15, 15) = T_{13}$, $(4, 16, 12, 12, 15, 15)$, $(9, 21, 10, 10, 12, 12)$                  \\
    \hline
    30  & $(6, 6, 10, 12, 15, 20) = T_{16}$, $(5, 7, 11, 11, 15, 20)$                                               \\
    60  & $(7, 17, 20, 24, 35, 35)$, $(7, 17, 22, 22, 33, 37)$,                                                     \\& $(10, 14, 17, 27, 35, 35)$, $(12, 12, 17, 27, 33, 37)$
    \\
    \hline
    30  & $(6, 10, 10, 15, 12, 18) = T_{21}$, $(5, 11, 10, 15, 13, 17)$                                             \\
    60  & $(10, 22, 21, 29, 25, 35)$, $(11, 21, 19, 31, 26, 34)$,                                                   \\& $(11, 21, 21, 29, 24, 36)$, $(12, 20, 19, 31, 25, 35)$
    \\
    \hline
    30  & $(6, 10, 6, 10, 15, 24) = T_6$                                                                            \\
    60  & $(7, 25, 12, 20, 35, 43)$                                                                                 \\
    \hline
    30  & $(6, 12, 6, 12, 15, 20) = T_2$                                                                            \\
    60  & $(12, 24, 13, 23, 29, 41)$                                                                                \\
    \hline
    30  & $(6, 12, 10, 10, 15, 18) = T_3$, $(7, 13, 9, 9, 15, 18)$                                                  \\
    60  & $(12, 24, 17, 23, 33, 33)$, $(14, 26, 15, 21, 33, 33)$,                                                   \\& $(15, 21, 20, 20, 27, 39)$, $(17, 23, 18, 18, 27, 39)$
    \\
    \hline
    30  & $(6, 15, 6, 18, 10, 20) = T_4$, $(6, 15, 7, 17, 9, 21)$                                                   \\
    60  & $(9, 33, 14, 34, 21, 39)$, $(9, 33, 15, 33, 20, 40)$,                                                     \\& $(11, 31, 12, 36, 21, 39)$, $(11, 31, 15, 33, 18, 42)$
    \\
    \hline
    30  & $(6, 15, 10, 15, 12, 15) = T_1$, $(6, 15, 11, 14, 11, 16)$, $(8, 13, 8, 17, 12, 15)$,                     \\& $(8, 13, 9, 18, 11, 14)$, $(8, 17, 9, 12, 11, 16)$, $(9, 12, 9, 18, 10, 15)$
    \\
    \hline
    30  & $(10, 12, 10, 12, 15, 12) = T_5$                                                                          \\
    60  & $(19, 25, 20, 24, 29, 25)$                                                                                \\
  \end{tabular}

  \caption{The $59$ sporadic rational tetrahedra.}
  \label{Table:tetrahedra}
\end{table}

\subsection{Maximal rational-angle \texorpdfstring{$n$}{n}-line configurations for \texorpdfstring{$5 \le n \le 15$}{n between 5 and 15}}
\label{S:15 to 5}

Each entry in the following list is a representative of
an $S_n^\pm$-orbit of $\R^3$-realizable $n \times n$ rational-angle matrices
for some $n \ge 5$,
with each angle measured in units of $\pi$.
The list is complete except that
we omit families whose general member is contained
in such an $N \times N$ matrix for some $N>n$.

  {\Small

    \noindent

    \smallskip
    \hrule
    \smallskip
    \centering
    $\begin{pmatrix}
        0   & 1/5 & 1/5 & 1/5 & 1/5 & 1/3 & 1/3 & 1/3 & 1/3 & 2/5 & 2/5 & 2/5 & 2/5 & 1/2 & 1/2 \\
        1/5 & 0   & 1/5 & 1/3 & 2/5 & 1/5 & 1/3 & 2/5 & 1/2 & 1/5 & 1/3 & 1/2 & 3/5 & 1/3 & 2/5 \\
        1/5 & 1/5 & 0   & 2/5 & 1/3 & 2/5 & 1/2 & 1/5 & 1/3 & 1/3 & 1/5 & 3/5 & 1/2 & 1/3 & 3/5 \\
        1/5 & 1/3 & 2/5 & 0   & 1/5 & 1/3 & 1/5 & 1/2 & 2/5 & 1/2 & 3/5 & 1/5 & 1/3 & 2/3 & 2/5 \\
        1/5 & 2/5 & 1/3 & 1/5 & 0   & 1/2 & 2/5 & 1/3 & 1/5 & 3/5 & 1/2 & 1/3 & 1/5 & 2/3 & 3/5 \\
        1/3 & 1/5 & 2/5 & 1/3 & 1/2 & 0   & 1/5 & 3/5 & 2/3 & 1/5 & 1/2 & 2/5 & 2/3 & 2/5 & 1/5 \\
        1/3 & 1/3 & 1/2 & 1/5 & 2/5 & 1/5 & 0   & 2/3 & 3/5 & 2/5 & 2/3 & 1/5 & 1/2 & 3/5 & 1/5 \\
        1/3 & 2/5 & 1/5 & 1/2 & 1/3 & 3/5 & 2/3 & 0   & 1/5 & 1/2 & 1/5 & 2/3 & 2/5 & 2/5 & 4/5 \\
        1/3 & 1/2 & 1/3 & 2/5 & 1/5 & 2/3 & 3/5 & 1/5 & 0   & 2/3 & 2/5 & 1/2 & 1/5 & 3/5 & 4/5 \\
        2/5 & 1/5 & 1/3 & 1/2 & 3/5 & 1/5 & 2/5 & 1/2 & 2/3 & 0   & 1/3 & 3/5 & 4/5 & 1/5 & 1/3 \\
        2/5 & 1/3 & 1/5 & 3/5 & 1/2 & 1/2 & 2/3 & 1/5 & 2/5 & 1/3 & 0   & 4/5 & 3/5 & 1/5 & 2/3 \\
        2/5 & 1/2 & 3/5 & 1/5 & 1/3 & 2/5 & 1/5 & 2/3 & 1/2 & 3/5 & 4/5 & 0   & 1/3 & 4/5 & 1/3 \\
        2/5 & 3/5 & 1/2 & 1/3 & 1/5 & 2/3 & 1/2 & 2/5 & 1/5 & 4/5 & 3/5 & 1/3 & 0   & 4/5 & 2/3 \\
        1/2 & 1/3 & 1/3 & 2/3 & 2/3 & 2/5 & 3/5 & 2/5 & 3/5 & 1/5 & 1/5 & 4/5 & 4/5 & 0   & 1/2 \\
        1/2 & 2/5 & 3/5 & 2/5 & 3/5 & 1/5 & 1/5 & 4/5 & 4/5 & 1/3 & 2/3 & 1/3 & 2/3 & 1/2 & 0   \\
      \end{pmatrix}$
    \\
    \smallskip
    \hrule
    \smallskip
    $\begin{pmatrix}
        0   & 1/4 & 1/4 & 1/4 & 1/4 & 1/2 & 1/2 & 1/2 & 1/2 \\
        1/4 & 0   & 1/3 & 1/3 & 1/2 & 1/4 & 1/3 & 1/3 & 1/2 \\
        1/4 & 1/3 & 0   & 1/2 & 1/3 & 1/2 & 1/3 & 2/3 & 1/4 \\
        1/4 & 1/3 & 1/2 & 0   & 1/3 & 1/2 & 2/3 & 1/3 & 3/4 \\
        1/4 & 1/2 & 1/3 & 1/3 & 0   & 3/4 & 2/3 & 2/3 & 1/2 \\
        1/2 & 1/4 & 1/2 & 1/2 & 3/4 & 0   & 1/4 & 1/4 & 1/2 \\
        1/2 & 1/3 & 1/3 & 2/3 & 2/3 & 1/4 & 0   & 1/2 & 1/4 \\
        1/2 & 1/3 & 2/3 & 1/3 & 2/3 & 1/4 & 1/2 & 0   & 3/4 \\
        1/2 & 1/2 & 1/4 & 3/4 & 1/2 & 1/2 & 1/4 & 3/4 & 0   \\
      \end{pmatrix}$
    \\
    \smallskip
    \hrule
    \smallskip
    $\begin{pmatrix}
        0     & 1/10  & 1/6   & 4/15  & 3/10  & 2/5   & 13/30 & 7/15  \\
        1/10  & 0     & 4/15  & 1/6   & 2/5   & 13/30 & 1/3   & 17/30 \\
        1/6   & 4/15  & 0     & 13/30 & 2/15  & 11/30 & 3/5   & 3/10  \\
        4/15  & 1/6   & 13/30 & 0     & 17/30 & 1/2   & 1/6   & 11/15 \\
        3/10  & 2/5   & 2/15  & 17/30 & 0     & 11/30 & 11/15 & 1/6   \\
        2/5   & 13/30 & 11/30 & 1/2   & 11/30 & 0     & 17/30 & 2/5   \\
        13/30 & 1/3   & 3/5   & 1/6   & 11/15 & 17/30 & 0     & 9/10  \\
        7/15  & 17/30 & 3/10  & 11/15 & 1/6   & 2/5   & 9/10  & 0     \\
      \end{pmatrix}$
    \\
    $\begin{pmatrix}
        0    & 1/15  & 1/15  & 3/10  & 3/10  & 7/15  & 7/15  & 1/2   \\
        1/15 & 0     & 2/15  & 7/30  & 11/30 & 2/5   & 8/15  & 7/15  \\
        1/15 & 2/15  & 0     & 11/30 & 7/30  & 8/15  & 2/5   & 8/15  \\
        3/10 & 7/30  & 11/30 & 0     & 3/5   & 1/6   & 23/30 & 11/30 \\
        3/10 & 11/30 & 7/30  & 3/5   & 0     & 23/30 & 1/6   & 19/30 \\
        7/15 & 2/5   & 8/15  & 1/6   & 23/30 & 0     & 14/15 & 1/3   \\
        7/15 & 8/15  & 2/5   & 23/30 & 1/6   & 14/15 & 0     & 2/3   \\
        1/2  & 7/15  & 8/15  & 11/30 & 19/30 & 1/3   & 2/3   & 0     \\
      \end{pmatrix}$
    \\
    \smallskip
    $\begin{pmatrix}
        0     & 1/15  & 1/5   & 4/15  & 4/15  & 11/30 & 13/30 & 7/15  \\
        1/15  & 0     & 4/15  & 1/5   & 4/15  & 13/30 & 11/30 & 8/15  \\
        1/5   & 4/15  & 0     & 7/15  & 1/3   & 1/6   & 19/30 & 4/15  \\
        4/15  & 1/5   & 7/15  & 0     & 1/3   & 19/30 & 1/6   & 11/15 \\
        4/15  & 4/15  & 1/3   & 1/3   & 0     & 13/30 & 13/30 & 1/2   \\
        11/30 & 13/30 & 1/6   & 19/30 & 13/30 & 0     & 4/5   & 1/10  \\
        13/30 & 11/30 & 19/30 & 1/6   & 13/30 & 4/5   & 0     & 9/10  \\
        7/15  & 8/15  & 4/15  & 11/15 & 1/2   & 1/10  & 9/10  & 0     \\
      \end{pmatrix}$
    \\
    \smallskip
    $\begin{pmatrix}
        0     & 1/15  & 1/10  & 1/6   & 1/5   & 3/10  & 11/30 & 7/15  \\
        1/15  & 0     & 1/6   & 1/10  & 1/5   & 11/30 & 3/10  & 8/15  \\
        1/10  & 1/6   & 0     & 4/15  & 7/30  & 1/5   & 7/15  & 11/30 \\
        1/6   & 1/10  & 4/15  & 0     & 7/30  & 7/15  & 1/5   & 19/30 \\
        1/5   & 1/5   & 7/30  & 7/30  & 0     & 11/30 & 11/30 & 1/2   \\
        3/10  & 11/30 & 1/5   & 7/15  & 11/30 & 0     & 2/3   & 1/6   \\
        11/30 & 3/10  & 7/15  & 1/5   & 11/30 & 2/3   & 0     & 5/6   \\
        7/15  & 8/15  & 11/30 & 19/30 & 1/2   & 1/6   & 5/6   & 0     \\
      \end{pmatrix}$
    \\
    \smallskip
    $\begin{pmatrix}
        0    & 1/10  & 2/15  & 3/10  & 3/10 & 1/3  & 2/5   & 7/15  \\
        1/10 & 0     & 7/30  & 1/5   & 2/5  & 3/10 & 3/10  & 17/30 \\
        2/15 & 7/30  & 0     & 13/30 & 1/6  & 2/5  & 8/15  & 1/3   \\
        3/10 & 1/5   & 13/30 & 0     & 3/5  & 3/10 & 1/10  & 23/30 \\
        3/10 & 2/5   & 1/6   & 3/5   & 0    & 1/2  & 7/10  & 1/6   \\
        1/3  & 3/10  & 2/5   & 3/10  & 1/2  & 0    & 1/3   & 3/5   \\
        2/5  & 3/10  & 8/15  & 1/10  & 7/10 & 1/3  & 0     & 13/15 \\
        7/15 & 17/30 & 1/3   & 23/30 & 1/6  & 3/5  & 13/15 & 0     \\
      \end{pmatrix}$
    \\
    \smallskip
    \hrule
    \smallskip
    $\begin{pmatrix}
        0         & 3/4 + x & 4 x       & 1/4 + 3 x & 1/2 + 2 x & 3/4 + x   \\
        3/4 + x   & 0       & 3/4 + x   & 2/3       & 1/2       & 1/3       \\
        4 x       & 3/4 + x & 0         & 1/4 - x   & 1/2 - 2 x & 3/4 - 3 x \\
        1/4 + 3 x & 2/3     & 1/4 - x   & 0         & 1/4 - x   & 1/2 - 2 x \\
        1/2 + 2 x & 1/2     & 1/2 - 2 x & 1/4 - x   & 0         & 1/4 - x   \\
        3/4 + x   & 1/3     & 3/4 - 3 x & 1/2 - 2 x & 1/4 - x   & 0         \\
      \end{pmatrix}
      \textup{ for } 0 \le x \le 1/12 $
    \\
    $\begin{pmatrix}
        0         & 2/3 + x & 2 x       & 1/2 + 2 x & 1/3 + x & 1/2 + x \\
        2/3 + x   & 0       & 2/3 - x   & 1/2 - x   & 1/3     & 1/6     \\
        2 x       & 2/3 - x & 0         & 1/2 + 2 x & 1/3 - x & 1/2 - x \\
        1/2 + 2 x & 1/2 - x & 1/2 + 2 x & 0         & 1/2 + x & 1/2     \\
        1/3 + x   & 1/3     & 1/3 - x   & 1/2 + x   & 0       & 1/6     \\
        1/2 + x   & 1/6     & 1/2 - x   & 1/2       & 1/6     & 0       \\
      \end{pmatrix}
      \textup{ for } 0 \le x \le 1/6 $
    \\
    \smallskip
    $\begin{pmatrix}
        0       & 1/2 + x & 2 x     & 2 x     & 1/2 + x & 1/2     \\
        1/2 + x & 0       & 1/2     & 1/2 - x & 1/3     & 1/3     \\
        2 x     & 1/2     & 0       & 2 x     & 1/2 - x & 1/2 + x \\
        2 x     & 1/2 - x & 2 x     & 0       & 1/2     & 1/2 - x \\
        1/2 + x & 1/3     & 1/2 - x & 1/2     & 0       & 2/3     \\
        1/2     & 1/3     & 1/2 + x & 1/2 - x & 2/3     & 0       \\
      \end{pmatrix}
      \textup{ for } 0 \le x \le 1/3 $
    \\
    \smallskip
    $\begin{pmatrix}
        0         & 1/2 + x & 4 x       & 1/2 + 3 x & 2 x     & 1/2 + x   \\
        1/2 + x   & 0       & 1/2 - x   & 1/3       & 1/2     & 1/3       \\
        4 x       & 1/2 - x & 0         & 1/2 - x   & 2 x     & 1/2 - 3 x \\
        1/2 + 3 x & 1/3     & 1/2 - x   & 0         & 1/2 + x & 2 x       \\
        2 x       & 1/2     & 2 x       & 1/2 + x   & 0       & 1/2 - x   \\
        1/2 + x   & 1/3     & 1/2 - 3 x & 2 x       & 1/2 - x & 0         \\
      \end{pmatrix}
      \textup{ for } 0 \le x \le 1/6 $
    \\
    \smallskip
    $\begin{pmatrix}
        0         & 2/3 + x   & 2/3 + 4 x & 3 x       & 1/3 + 2 x & 2/3 + x \\
        2/3 + x   & 0         & 3 x       & 2/3 - 2 x & 1/3 - x   & 1/3     \\
        2/3 + 4 x & 3 x       & 0         & 2/3 + x   & 1/3 + 2 x & 1/3 - x \\
        3 x       & 2/3 - 2 x & 2/3 + x   & 0         & 1/3 - x   & 2/3     \\
        1/3 + 2 x & 1/3 - x   & 1/3 + 2 x & 1/3 - x   & 0         & 1/2     \\
        2/3 + x   & 1/3       & 1/3 - x   & 2/3       & 1/2       & 0       \\
      \end{pmatrix}
      \textup{ for } 0 \le x \le 1/12 $
    \\
    \smallskip
    $\begin{pmatrix}
        0     & 1/21  & 5/42  & 1/6   & 2/7   & 10/21 \\
        1/21  & 0     & 1/6   & 5/42  & 2/7   & 11/21 \\
        5/42  & 1/6   & 0     & 2/7   & 13/42 & 5/14  \\
        1/6   & 5/42  & 2/7   & 0     & 13/42 & 9/14  \\
        2/7   & 2/7   & 13/42 & 13/42 & 0     & 1/2   \\
        10/21 & 11/21 & 5/14  & 9/14  & 1/2   & 0     \\
      \end{pmatrix}$, \hspace*{0.25in}
    $\begin{pmatrix}
        0    & 1/14  & 1/14  & 8/21  & 8/21  & 1/2   \\
        1/14 & 0     & 1/7   & 13/42 & 19/42 & 19/42 \\
        1/14 & 1/7   & 0     & 19/42 & 13/42 & 23/42 \\
        8/21 & 13/42 & 19/42 & 0     & 16/21 & 2/7   \\
        8/21 & 19/42 & 13/42 & 16/21 & 0     & 5/7   \\
        1/2  & 19/42 & 23/42 & 2/7   & 5/7   & 0     \\
      \end{pmatrix}$
    \\
    \smallskip
    $\begin{pmatrix}
        0    & 2/15  & 2/15  & 3/10  & 3/10  & 1/2   \\
        2/15 & 0     & 4/15  & 1/6   & 13/30 & 7/15  \\
        2/15 & 4/15  & 0     & 13/30 & 1/6   & 8/15  \\
        3/10 & 1/6   & 13/30 & 0     & 3/5   & 13/30 \\
        3/10 & 13/30 & 1/6   & 3/5   & 0     & 17/30 \\
        1/2  & 7/15  & 8/15  & 13/30 & 17/30 & 0     \\
      \end{pmatrix}$, \hspace*{0.25in}
    $\begin{pmatrix}
        0     & 1/21  & 1/7   & 11/42 & 13/42 & 10/21 \\
        1/21  & 0     & 1/7   & 13/42 & 11/42 & 11/21 \\
        1/7   & 1/7   & 0     & 13/42 & 13/42 & 1/2   \\
        11/42 & 13/42 & 13/42 & 0     & 4/7   & 3/14  \\
        13/42 & 11/42 & 13/42 & 4/7   & 0     & 11/14 \\
        10/21 & 11/21 & 1/2   & 3/14  & 11/14 & 0     \\
      \end{pmatrix}$
    \\
    \smallskip
    $\begin{pmatrix}
        0     & 1/10 & 1/10 & 11/30 & 11/30 & 1/2  \\
        1/10  & 0    & 1/5  & 4/15  & 7/15  & 7/15 \\
        1/10  & 1/5  & 0    & 7/15  & 4/15  & 8/15 \\
        11/30 & 4/15 & 7/15 & 0     & 11/15 & 2/5  \\
        11/30 & 7/15 & 4/15 & 11/15 & 0     & 3/5  \\
        1/2   & 7/15 & 8/15 & 2/5   & 3/5   & 0    \\
      \end{pmatrix}$, \hspace*{0.25in}
    $\begin{pmatrix}
        0    & 1/14  & 2/21  & 2/7   & 5/14  & 8/21  \\
        1/14 & 0     & 1/6   & 11/42 & 2/7   & 19/42 \\
        2/21 & 1/6   & 0     & 1/3   & 19/42 & 2/7   \\
        2/7  & 11/42 & 1/3   & 0     & 13/42 & 11/21 \\
        5/14 & 2/7   & 19/42 & 13/42 & 0     & 31/42 \\
        8/21 & 19/42 & 2/7   & 11/21 & 31/42 & 0     \\
      \end{pmatrix}$
    \\
    \smallskip
    $\begin{pmatrix}
        0     & 1/15  & 7/60  & 11/60 & 7/20  & 7/15  \\
        1/15  & 0     & 11/60 & 7/60  & 7/20  & 8/15  \\
        7/60  & 11/60 & 0     & 3/10  & 11/30 & 7/20  \\
        11/60 & 7/60  & 3/10  & 0     & 11/30 & 13/20 \\
        7/20  & 7/20  & 11/30 & 11/30 & 0     & 1/2   \\
        7/15  & 8/15  & 7/20  & 13/20 & 1/2   & 0     \\
      \end{pmatrix}$, \hspace*{0.25in}
    $\begin{pmatrix}
        0     & 2/15  & 2/15  & 13/30 & 13/30 & 1/2   \\
        2/15  & 0     & 4/15  & 3/10  & 17/30 & 2/5   \\
        2/15  & 4/15  & 0     & 17/30 & 3/10  & 3/5   \\
        13/30 & 3/10  & 17/30 & 0     & 13/15 & 7/30  \\
        13/30 & 17/30 & 3/10  & 13/15 & 0     & 23/30 \\
        1/2   & 2/5   & 3/5   & 7/30  & 23/30 & 0     \\
      \end{pmatrix}$
    \\
    \smallskip
    $\begin{pmatrix}
        0     & 1/6   & 3/14  & 17/42 & 3/7   & 19/42 \\
        1/6   & 0     & 8/21  & 5/21  & 25/42 & 3/7   \\
        3/14  & 8/21  & 0     & 13/21 & 3/14  & 1/2   \\
        17/42 & 5/21  & 13/21 & 0     & 5/6   & 3/7   \\
        3/7   & 25/42 & 3/14  & 5/6   & 0     & 23/42 \\
        19/42 & 3/7   & 1/2   & 3/7   & 23/42 & 0     \\
      \end{pmatrix}$, \hspace*{0.25in}
    $\begin{pmatrix}
        0     & 2/15  & 1/6   & 4/15  & 1/3   & 13/30 \\
        2/15  & 0     & 1/6   & 1/3   & 4/15  & 17/30 \\
        1/6   & 1/6   & 0     & 13/30 & 13/30 & 1/2   \\
        4/15  & 1/3   & 13/30 & 0     & 4/15  & 11/30 \\
        1/3   & 4/15  & 13/30 & 4/15  & 0     & 19/30 \\
        13/30 & 17/30 & 1/2   & 11/30 & 19/30 & 0     \\
      \end{pmatrix}$
    \\
    \smallskip
    $\begin{pmatrix}
        0     & 7/60  & 1/6   & 17/60 & 2/5  & 5/12 \\
        7/60  & 0     & 17/60 & 2/5   & 5/12 & 3/10 \\
        1/6   & 17/60 & 0     & 7/60  & 2/5  & 7/12 \\
        17/60 & 2/5   & 7/60  & 0     & 5/12 & 7/10 \\
        2/5   & 5/12  & 2/5   & 5/12  & 0    & 1/2  \\
        5/12  & 3/10  & 7/12  & 7/10  & 1/2  & 0    \\
      \end{pmatrix}$, \hspace*{0.25in}
    $\begin{pmatrix}
        0    & 1/15  & 1/15  & 9/20  & 9/20  & 1/2   \\
        1/15 & 0     & 2/15  & 23/60 & 31/60 & 9/20  \\
        1/15 & 2/15  & 0     & 31/60 & 23/60 & 11/20 \\
        9/20 & 23/60 & 31/60 & 0     & 9/10  & 7/30  \\
        9/20 & 31/60 & 23/60 & 9/10  & 0     & 23/30 \\
        1/2  & 9/20  & 11/20 & 7/30  & 23/30 & 0     \\
      \end{pmatrix}$
    \\
    \smallskip
    $\begin{pmatrix}
        0     & 1/6   & 1/6   & 11/30 & 11/30 & 1/2   \\
        1/6   & 0     & 4/15  & 1/3   & 8/15  & 11/30 \\
        1/6   & 4/15  & 0     & 8/15  & 1/3   & 19/30 \\
        11/30 & 1/3   & 8/15  & 0     & 8/15  & 7/30  \\
        11/30 & 8/15  & 1/3   & 8/15  & 0     & 23/30 \\
        1/2   & 11/30 & 19/30 & 7/30  & 23/30 & 0     \\
      \end{pmatrix}$, \hspace*{0.25in}
    $\begin{pmatrix}
        0     & 2/15  & 2/15 & 1/5  & 4/15 & 13/30 \\
        2/15  & 0     & 4/15 & 1/5  & 2/15 & 17/30 \\
        2/15  & 4/15  & 0    & 4/15 & 2/5  & 3/10  \\
        1/5   & 1/5   & 4/15 & 0    & 4/15 & 1/2   \\
        4/15  & 2/15  & 2/5  & 4/15 & 0    & 7/10  \\
        13/30 & 17/30 & 3/10 & 1/2  & 7/10 & 0     \\
      \end{pmatrix}$
    \\
    \smallskip
    $\begin{pmatrix}
        0     & 2/21  & 3/14  & 5/14  & 3/7   & 10/21 \\
        2/21  & 0     & 13/42 & 11/42 & 8/21  & 4/7   \\
        3/14  & 13/42 & 0     & 4/7   & 23/42 & 11/42 \\
        5/14  & 11/42 & 4/7   & 0     & 13/42 & 5/6   \\
        3/7   & 8/21  & 23/42 & 13/42 & 0     & 2/3   \\
        10/21 & 4/7   & 11/42 & 5/6   & 2/3   & 0     \\
      \end{pmatrix}$, \hspace*{0.25in}
    $\begin{pmatrix}
        0     & 1/15  & 3/20  & 19/60 & 23/60 & 7/15  \\
        1/15  & 0     & 3/20  & 23/60 & 19/60 & 8/15  \\
        3/20  & 3/20  & 0     & 11/30 & 11/30 & 1/2   \\
        19/60 & 23/60 & 11/30 & 0     & 7/10  & 3/20  \\
        23/60 & 19/60 & 11/30 & 7/10  & 0     & 17/20 \\
        7/15  & 8/15  & 1/2   & 3/20  & 17/20 & 0     \\
      \end{pmatrix}$
    \\
    \smallskip
    $\begin{pmatrix}
        0     & 1/10  & 11/60 & 17/60 & 13/30 & 9/20  \\
        1/10  & 0     & 17/60 & 11/60 & 13/30 & 11/20 \\
        11/60 & 17/60 & 0     & 7/15  & 9/20  & 4/15  \\
        17/60 & 11/60 & 7/15  & 0     & 9/20  & 11/15 \\
        13/30 & 13/30 & 9/20  & 9/20  & 0     & 1/2   \\
        9/20  & 11/20 & 4/15  & 11/15 & 1/2   & 0     \\
      \end{pmatrix}$, \hspace*{0.25in}
    $\begin{pmatrix}
        0    & 2/21  & 2/21  & 5/14  & 5/14  & 1/2   \\
        2/21 & 0     & 4/21  & 11/42 & 19/42 & 3/7   \\
        2/21 & 4/21  & 0     & 19/42 & 11/42 & 4/7   \\
        5/14 & 11/42 & 19/42 & 0     & 5/7   & 11/42 \\
        5/14 & 19/42 & 11/42 & 5/7   & 0     & 31/42 \\
        1/2  & 3/7   & 4/7   & 11/42 & 31/42 & 0     \\
      \end{pmatrix}$
    \\
    \smallskip
    $\begin{pmatrix}
        0    & 1/10  & 1/10  & 5/12  & 5/12  & 1/2  \\
        1/10 & 0     & 1/5   & 19/60 & 31/60 & 5/12 \\
        1/10 & 1/5   & 0     & 31/60 & 19/60 & 7/12 \\
        5/12 & 19/60 & 31/60 & 0     & 5/6   & 1/5  \\
        5/12 & 31/60 & 19/60 & 5/6   & 0     & 4/5  \\
        1/2  & 5/12  & 7/12  & 1/5   & 4/5   & 0    \\
      \end{pmatrix}$, \hspace*{0.25in}
    $\begin{pmatrix}
        0    & 1/15 & 7/30  & 3/10  & 2/5   & 7/15  \\
        1/15 & 0    & 3/10  & 7/30  & 2/5   & 8/15  \\
        7/30 & 3/10 & 0     & 8/15  & 13/30 & 7/30  \\
        3/10 & 7/30 & 8/15  & 0     & 13/30 & 23/30 \\
        2/5  & 2/5  & 13/30 & 13/30 & 0     & 1/2   \\
        7/15 & 8/15 & 7/30  & 23/30 & 1/2   & 0     \\
      \end{pmatrix}$
    \\
    \smallskip
    $\begin{pmatrix}
        0     & 5/42  & 5/21 & 17/42 & 3/7   & 10/21 \\
        5/42  & 0     & 5/14 & 2/7   & 17/42 & 25/42 \\
        5/21  & 5/14  & 0    & 9/14  & 1/2   & 5/21  \\
        17/42 & 2/7   & 9/14 & 0     & 17/42 & 37/42 \\
        3/7   & 17/42 & 1/2  & 17/42 & 0     & 4/7   \\
        10/21 & 25/42 & 5/21 & 37/42 & 4/7   & 0     \\
      \end{pmatrix}$, \hspace*{0.25in}
    $\begin{pmatrix}
        0    & 1/15 & 2/15 & 1/6  & 7/30 & 7/15 \\
        1/15 & 0    & 2/15 & 7/30 & 1/6  & 8/15 \\
        2/15 & 2/15 & 0    & 7/30 & 7/30 & 1/2  \\
        1/6  & 7/30 & 7/30 & 0    & 2/5  & 3/10 \\
        7/30 & 1/6  & 7/30 & 2/5  & 0    & 7/10 \\
        7/15 & 8/15 & 1/2  & 3/10 & 7/10 & 0    \\
      \end{pmatrix}$
    \\
    \smallskip
    \hrule
    \smallskip
    $\begin{pmatrix}
        0         & 2/3 + x & 1/6 + 2 x & 1/3 + 2 x & 1/2 + 2 x \\
        2/3 + x   & 0       & 5/6 - x   & 1/3 - x   & 1/6 + x   \\
        1/6 + 2 x & 5/6 - x & 0         & 1/2 + 2 x & 2/3 - 2 x \\
        1/3 + 2 x & 1/3 - x & 1/2 + 2 x & 0         & 1/6 + 2 x \\
        1/2 + 2 x & 1/6 + x & 2/3 - 2 x & 1/6 + 2 x & 0         \\
      \end{pmatrix}
      \textup{ for } 0 \le x \le 1/6 $
    \\
    \smallskip
    $\begin{pmatrix}
        0         & 2/3 + x   & 1/6 + 3 x & 1/6 + 2 x & 1/2 + x \\
        2/3 + x   & 0         & 5/6 - 2 x & 1/2 - x   & 1/6     \\
        1/6 + 3 x & 5/6 - 2 x & 0         & 1/3 + 3 x & 2/3     \\
        1/6 + 2 x & 1/2 - x   & 1/3 + 3 x & 0         & 1/3 - x \\
        1/2 + x   & 1/6       & 2/3       & 1/3 - x   & 0       \\
      \end{pmatrix}
      \textup{ for } 0 \le x \le 1/6 $
    \\
    \smallskip
    $\begin{pmatrix}
        0     & 7/60 & 13/60 & 7/30  & 7/20 \\
        7/60  & 0    & 7/30  & 7/20  & 7/15 \\
        13/60 & 7/30 & 0     & 19/60 & 2/5  \\
        7/30  & 7/20 & 19/60 & 0     & 7/60 \\
        7/20  & 7/15 & 2/5   & 7/60  & 0    \\
      \end{pmatrix}$, \hspace*{0.25in}
    $\begin{pmatrix}
        0     & 4/21  & 2/7   & 17/42 & 3/7   \\
        4/21  & 0     & 3/7   & 25/42 & 2/7   \\
        2/7   & 3/7   & 0     & 11/42 & 10/21 \\
        17/42 & 25/42 & 11/42 & 0     & 31/42 \\
        3/7   & 2/7   & 10/21 & 31/42 & 0     \\
      \end{pmatrix}$
    \\
    \smallskip
    $\begin{pmatrix}
        0    & 1/10 & 1/3   & 7/20 & 9/20  \\
        1/10 & 0    & 2/5   & 9/20 & 7/20  \\
        1/3  & 2/5  & 0     & 1/4  & 13/20 \\
        7/20 & 9/20 & 1/4   & 0    & 4/5   \\
        9/20 & 7/20 & 13/20 & 4/5  & 0     \\
      \end{pmatrix}$, \hspace*{0.25in}
    $\begin{pmatrix}
        0     & 1/7   & 4/21  & 5/21 & 19/42 \\
        1/7   & 0     & 5/21  & 8/21 & 25/42 \\
        4/21  & 5/21  & 0     & 2/7  & 19/42 \\
        5/21  & 8/21  & 2/7   & 0    & 3/14  \\
        19/42 & 25/42 & 19/42 & 3/14 & 0     \\
      \end{pmatrix}$
    \\
    \smallskip
    $\begin{pmatrix}
        0    & 1/5  & 2/5  & 2/5  & 7/15 \\
        1/5  & 0    & 8/15 & 3/5  & 3/5  \\
        2/5  & 8/15 & 0    & 4/15 & 8/15 \\
        2/5  & 3/5  & 4/15 & 0    & 4/15 \\
        7/15 & 3/5  & 8/15 & 4/15 & 0    \\
      \end{pmatrix}$, \hspace*{0.25in}
    $\begin{pmatrix}
        0   & 1/7 & 1/3 & 3/7 & 3/7 \\
        1/7 & 0   & 3/7 & 1/3 & 4/7 \\
        1/3 & 3/7 & 0   & 3/7 & 2/7 \\
        3/7 & 1/3 & 3/7 & 0   & 5/7 \\
        3/7 & 4/7 & 2/7 & 5/7 & 0   \\
      \end{pmatrix}$
    \\
    \smallskip
    $\begin{pmatrix}
        0     & 1/14  & 5/21  & 3/7   & 10/21 \\
        1/14  & 0     & 13/42 & 17/42 & 17/42 \\
        5/21  & 13/42 & 0     & 11/21 & 5/7   \\
        3/7   & 17/42 & 11/21 & 0     & 8/21  \\
        10/21 & 17/42 & 5/7   & 8/21  & 0     \\
      \end{pmatrix}$, \hspace*{0.25in}
    $\begin{pmatrix}
        0     & 1/6   & 11/60 & 11/60 & 7/20  \\
        1/6   & 0     & 17/60 & 7/20  & 11/60 \\
        11/60 & 17/60 & 0     & 1/5   & 13/30 \\
        11/60 & 7/20  & 1/5   & 0     & 8/15  \\
        7/20  & 11/60 & 13/30 & 8/15  & 0     \\
      \end{pmatrix}$
    \\
    \smallskip
    $\begin{pmatrix}
        0    & 2/15 & 2/15 & 1/5  & 1/5  \\
        2/15 & 0    & 4/15 & 1/5  & 4/15 \\
        2/15 & 4/15 & 0    & 4/15 & 1/5  \\
        1/5  & 1/5  & 4/15 & 0    & 2/5  \\
        1/5  & 4/15 & 1/5  & 2/5  & 0    \\
      \end{pmatrix}$, \hspace*{0.25in}
    $\begin{pmatrix}
        0    & 2/15  & 3/10  & 1/3  & 2/5  \\
        2/15 & 0     & 13/30 & 2/5  & 1/2  \\
        3/10 & 13/30 & 0     & 3/10 & 7/30 \\
        1/3  & 2/5   & 3/10  & 0    & 8/15 \\
        2/5  & 1/2   & 7/30  & 8/15 & 0    \\
      \end{pmatrix}$
    \\
    \smallskip
    $\begin{pmatrix}
        0     & 11/60 & 1/5   & 4/15  & 9/20  \\
        11/60 & 0     & 13/60 & 9/20  & 19/30 \\
        1/5   & 13/60 & 0     & 11/30 & 31/60 \\
        4/15  & 9/20  & 11/30 & 0     & 11/60 \\
        9/20  & 19/30 & 31/60 & 11/60 & 0     \\
      \end{pmatrix}$, \hspace*{0.25in}
    $\begin{pmatrix}
        0     & 1/10  & 19/60 & 11/30 & 5/12  \\
        1/10  & 0     & 5/12  & 13/30 & 19/60 \\
        19/60 & 5/12  & 0     & 1/4   & 11/15 \\
        11/30 & 13/30 & 1/4   & 0     & 13/20 \\
        5/12  & 19/60 & 11/15 & 13/20 & 0     \\
      \end{pmatrix}$
    \\
    \smallskip
    $\begin{pmatrix}
        0    & 3/20 & 3/20 & 7/20 & 7/20 \\
        3/20 & 0    & 3/10 & 1/3  & 2/5  \\
        3/20 & 3/10 & 0    & 2/5  & 1/3  \\
        7/20 & 1/3  & 2/5  & 0    & 7/10 \\
        7/20 & 2/5  & 1/3  & 7/10 & 0    \\
      \end{pmatrix}$, \hspace*{0.25in}
    $\begin{pmatrix}
        0     & 3/20  & 11/60 & 7/30  & 23/60 \\
        3/20  & 0     & 7/30  & 23/60 & 8/15  \\
        11/60 & 7/30  & 0     & 17/60 & 2/5   \\
        7/30  & 23/60 & 17/60 & 0     & 3/20  \\
        23/60 & 8/15  & 2/5   & 3/20  & 0     \\
      \end{pmatrix}$
    \\
    \smallskip
    $\begin{pmatrix}
        0     & 1/10 & 1/6   & 4/15  & 11/30 \\
        1/10  & 0    & 4/15  & 7/30  & 4/15  \\
        1/6   & 4/15 & 0     & 11/30 & 8/15  \\
        4/15  & 7/30 & 11/30 & 0     & 3/10  \\
        11/30 & 4/15 & 8/15  & 3/10  & 0     \\
      \end{pmatrix}$, \hspace*{0.25in}
    $\begin{pmatrix}
        0    & 2/15 & 1/6   & 1/6  & 7/15  \\
        2/15 & 0    & 7/30  & 3/10 & 3/5   \\
        1/6  & 7/30 & 0     & 1/5  & 13/30 \\
        1/6  & 3/10 & 1/5   & 0    & 3/10  \\
        7/15 & 3/5  & 13/30 & 3/10 & 0     \\
      \end{pmatrix}$
    \\
    \smallskip
    $\begin{pmatrix}
        0     & 2/15  & 3/20  & 17/60 & 2/5   \\
        2/15  & 0     & 17/60 & 3/20  & 13/30 \\
        3/20  & 17/60 & 0     & 13/30 & 23/60 \\
        17/60 & 3/20  & 13/30 & 0     & 29/60 \\
        2/5   & 13/30 & 23/60 & 29/60 & 0     \\
      \end{pmatrix}$, \hspace*{0.25in}
    $\begin{pmatrix}
        0     & 1/15  & 11/30 & 23/60 & 9/20  \\
        1/15  & 0     & 2/5   & 9/20  & 23/60 \\
        11/30 & 2/5   & 0     & 19/60 & 37/60 \\
        23/60 & 9/20  & 19/60 & 0     & 5/6   \\
        9/20  & 23/60 & 37/60 & 5/6   & 0     \\
      \end{pmatrix}$
    \\
    \smallskip
    $\begin{pmatrix}
        0    & 1/18 & 1/9   & 1/6 & 5/18  \\
        1/18 & 0    & 5/18  & 1/9 & 2/9   \\
        1/9  & 5/18 & 0     & 1/2 & 13/18 \\
        1/6  & 1/9  & 1/2   & 0   & 1/9   \\
        5/18 & 2/9  & 13/18 & 1/9 & 0     \\
      \end{pmatrix}$, \hspace*{0.25in}
    $\begin{pmatrix}
        0    & 7/60 & 7/30  & 3/10  & 5/12 \\
        7/60 & 0    & 1/4   & 5/12  & 8/15 \\
        7/30 & 1/4  & 0     & 11/30 & 9/20 \\
        3/10 & 5/12 & 11/30 & 0     & 7/60 \\
        5/12 & 8/15 & 9/20  & 7/60  & 0    \\
      \end{pmatrix}$
    \\
    \smallskip
    $\begin{pmatrix}
        0     & 2/15  & 1/5   & 7/20  & 29/60 \\
        2/15  & 0     & 7/30  & 29/60 & 7/20  \\
        1/5   & 7/30  & 0     & 23/60 & 29/60 \\
        7/20  & 29/60 & 23/60 & 0     & 5/6   \\
        29/60 & 7/20  & 29/60 & 5/6   & 0     \\
      \end{pmatrix}$, \hspace*{0.25in}
    $\begin{pmatrix}
        0    & 2/15  & 7/30 & 3/10  & 1/2   \\
        2/15 & 0     & 7/30 & 13/30 & 7/15  \\
        7/30 & 7/30  & 0    & 2/5   & 7/10  \\
        3/10 & 13/30 & 2/5  & 0     & 17/30 \\
        1/2  & 7/15  & 7/10 & 17/30 & 0     \\
      \end{pmatrix}$
    \\
    \smallskip
    $\begin{pmatrix}
        0    & 3/20 & 1/5  & 3/10 & 9/20 \\
        3/20 & 0    & 1/4  & 9/20 & 3/5  \\
        1/5  & 1/4  & 0    & 1/3  & 9/20 \\
        3/10 & 9/20 & 1/3  & 0    & 3/20 \\
        9/20 & 3/5  & 9/20 & 3/20 & 0    \\
      \end{pmatrix}$, \hspace*{0.25in}
    $\begin{pmatrix}
        0   & 1/7 & 1/7 & 2/7 & 2/7 \\
        1/7 & 0   & 2/7 & 2/7 & 1/3 \\
        1/7 & 2/7 & 0   & 1/3 & 2/7 \\
        2/7 & 2/7 & 1/3 & 0   & 4/7 \\
        2/7 & 1/3 & 2/7 & 4/7 & 0   \\
      \end{pmatrix}$
    \\
    \smallskip
    $\begin{pmatrix}
        0     & 3/20  & 1/6   & 13/60 & 19/60 \\
        3/20  & 0     & 19/60 & 1/5   & 7/15  \\
        1/6   & 19/60 & 0     & 19/60 & 3/20  \\
        13/60 & 1/5   & 19/60 & 0     & 13/30 \\
        19/60 & 7/15  & 3/20  & 13/30 & 0     \\
      \end{pmatrix}$, \hspace*{0.25in}
    $\begin{pmatrix}
        0     & 2/15  & 2/5  & 13/30 & 1/2   \\
        2/15  & 0     & 1/2  & 17/30 & 2/5   \\
        2/5   & 1/2   & 0    & 7/30  & 8/15  \\
        13/30 & 17/30 & 7/30 & 0     & 23/30 \\
        1/2   & 2/5   & 8/15 & 23/30 & 0     \\
      \end{pmatrix}$
    \\
    \smallskip
    $\begin{pmatrix}
        0    & 1/15  & 7/20  & 2/5   & 5/12  \\
        1/15 & 0     & 5/12  & 13/30 & 7/20  \\
        7/20 & 5/12  & 0     & 19/60 & 23/30 \\
        2/5  & 13/30 & 19/60 & 0     & 37/60 \\
        5/12 & 7/20  & 23/30 & 37/60 & 0     \\
      \end{pmatrix}$, \hspace*{0.25in}
    $\begin{pmatrix}
        0    & 2/15 & 1/6  & 1/6   & 7/15  \\
        2/15 & 0    & 7/30 & 3/10  & 1/2   \\
        1/6  & 7/30 & 0    & 1/5   & 3/10  \\
        1/6  & 3/10 & 1/5  & 0     & 13/30 \\
        7/15 & 1/2  & 3/10 & 13/30 & 0     \\
      \end{pmatrix}$
    \\
    \smallskip
    $\begin{pmatrix}
        0     & 3/20 & 13/60 & 5/12  & 13/30 \\
        3/20  & 0    & 7/30  & 4/15  & 7/12  \\
        13/60 & 7/30 & 0     & 2/5   & 29/60 \\
        5/12  & 4/15 & 2/5   & 0     & 17/20 \\
        13/30 & 7/12 & 29/60 & 17/20 & 0     \\
      \end{pmatrix}$
    \\

  }

\subsection{Regge orbits of rational-angle 4-line configurations}

Here we list representatives of the $\Regge^{\pm}$-orbits
of $\R^3$-realizable $4 \times 4$ rational-angle matrices,
with each angle measured in units of $\pi$.
We exclude any orbit containing a representative
obtained from a matrix in Section~\ref{S:15 to 5}
or from a perpendicular configuration.
In each orbit, we choose the representative
for which the sum of the denominators of the matrix entries is smallest,
and among those, we choose the one that is lexicographically smallest.

\smallskip
\hrule
\smallskip

{\Small
  \centering

  $\begin{pmatrix}
      0     & 1/20 & 11/60 & 11/30 \\
      1/20  & 0    & 1/5   & 5/12  \\
      11/60 & 1/5  & 0     & 7/20  \\
      11/30 & 5/12 & 7/20  & 0     \\
    \end{pmatrix}$, \hspace*{0.1in}
  $\begin{pmatrix}
      0     & 1/14  & 13/84 & 25/84 \\
      1/14  & 0     & 19/84 & 9/28  \\
      13/84 & 19/84 & 0     & 2/7   \\
      25/84 & 9/28  & 2/7   & 0     \\
    \end{pmatrix}$, \hspace*{0.1in}
  $\begin{pmatrix}
      0     & 11/90 & 4/15 & 22/45 \\
      11/90 & 0     & 7/18 & 17/30 \\
      4/15  & 7/18  & 0    & 1/3   \\
      22/45 & 17/30 & 1/3  & 0     \\
    \end{pmatrix}$
  \\
  \medskip
  $\begin{pmatrix}
      0      & 3/20   & 19/120 & 5/24  \\
      3/20   & 0      & 37/120 & 13/40 \\
      19/120 & 37/120 & 0      & 3/20  \\
      5/24   & 13/40  & 3/20   & 0     \\
    \end{pmatrix}$, \hspace*{0.05in}
  $\begin{pmatrix}
      0    & 1/15  & 5/18  & 1/3   \\
      1/15 & 0     & 31/90 & 17/45 \\
      5/18 & 31/90 & 0     & 7/30  \\
      1/3  & 17/45 & 7/30  & 0     \\
    \end{pmatrix}$
  \\
  \medskip
  $\begin{pmatrix}
      0      & 2/15   & 3/14   & 44/105 \\
      2/15   & 0      & 73/210 & 3/7    \\
      3/14   & 73/210 & 0      & 13/30  \\
      44/105 & 3/7    & 13/30  & 0      \\
    \end{pmatrix}$, \hspace*{0.1in}
  $\begin{pmatrix}
      0    & 2/35 & 2/7   & 3/10  \\
      2/35 & 0    & 1/3   & 5/14  \\
      2/7  & 1/3  & 0     & 11/70 \\
      3/10 & 5/14 & 11/70 & 0     \\
    \end{pmatrix}$
  \\
  \medskip
  $\begin{pmatrix}
      0      & 7/120  & 3/10   & 43/120 \\
      7/120  & 0      & 43/120 & 2/5    \\
      3/10   & 43/120 & 0      & 9/40   \\
      43/120 & 2/5    & 9/40   & 0      \\
    \end{pmatrix}$, \hspace*{0.1in}
  $\begin{pmatrix}
      0    & 1/18  & 1/3   & 7/15  \\
      1/18 & 0     & 11/30 & 37/90 \\
      1/3  & 11/30 & 0     & 29/45 \\
      7/15 & 37/90 & 29/45 & 0     \\
    \end{pmatrix}$
  \\
  \medskip
  $\begin{pmatrix}
      0     & 1/10  & 3/14  & 27/70 \\
      1/10  & 0     & 11/35 & 3/7   \\
      3/14  & 11/35 & 0     & 1/3   \\
      27/70 & 3/7   & 1/3   & 0     \\
    \end{pmatrix}$, \hspace*{0.1in}
  $\begin{pmatrix}
      0    & 5/84 & 1/4   & 5/14  \\
      5/84 & 0    & 2/7   & 5/12  \\
      1/4  & 2/7  & 0     & 23/84 \\
      5/14 & 5/12 & 23/84 & 0     \\
    \end{pmatrix}$
  \\
  \medskip
  $\begin{pmatrix}
      0     & 3/14  & 3/10  & 29/70 \\
      3/14  & 0     & 18/35 & 1/3   \\
      3/10  & 18/35 & 0     & 4/7   \\
      29/70 & 1/3   & 4/7   & 0     \\
    \end{pmatrix}$, \hspace*{0.1in}
  $\begin{pmatrix}
      0    & 8/45  & 3/10  & 1/3   \\
      8/45 & 0     & 43/90 & 2/5   \\
      3/10 & 43/90 & 0     & 31/90 \\
      1/3  & 2/5   & 31/90 & 0     \\
    \end{pmatrix}$

}

\section*{Acknowledgments}

The authors thank Jack Huizenga, Igor Rivin, 
and Justin Roberts for helpful discussions.

\end{document}